\documentclass[final]{amsart}   

\usepackage{amsmath,amssymb,latexsym,amscd}

\usepackage{graphicx,xcolor,mathrsfs}
\usepackage{subcaption}
\usepackage{empheq}

\renewcommand{\geq}{\geqslant}
\renewcommand{\leq}{\leqslant}

\newcommand{\RM}{\mathbb{R}}
\newcommand{\ZM}{\mathbb{Z}}
\newcommand{\NM}{\mathbb{N}}

\newcommand{\bigO}{\mathcal{O}}

\DeclareMathOperator{\supp}{supp}

\newtheorem{theorem}{Theorem}[section]
\newtheorem{lemma}[theorem]{Lemma}
\newtheorem{proposition}[theorem]{Proposition}
\newtheorem{remark}[theorem]{Remark}
\newtheorem{definition}[theorem]{Definition}

\newtheorem*{main-theorem}{Main Theorem}
\newtheorem*{remark*}{Remark}
\numberwithin{equation}{section}

\title[Existence of a Highest Wave]{Existence of a highest wave in a fully dispersive two-way shallow water model}

\author[Ehrnstr\"om]{Mats Ehrnstr\"om}
\address{Department of Mathematical Sciences, NTNU Norwegian University of Science and Technology, 7491 Trondheim, Norway}
\email{mats.ehrnstrom@math.ntnu.no}

\author[Johnson]{Mathew~A.~Johnson}
\address{Department of Mathematics, University of Kansas, Lawrence, KS 66045 USA} 
\email{matjohn@ku.edu}

\author[Claassen]{Kyle M. Claassen}
\address{Department of Mathematics, University of Kansas, Lawrence, KS 66045 USA} 
\email{kclaassen@ku.edu}  
\thanks{ME was supported by grant nos. 231668 and 250070 from the Research Council of Norway; 
MJ was supported by the National Science Foundation under grants DMS-1614785 and DMS-1211183; 
KC was supported by the National Science Foundation under grant DMS-1211183.
}

\date{\today}

\keywords{Highest wave; singular solutions; Whitham equation; water waves}
\subjclass[2010]{35Q35 (primary), 35B65, 37K50, 76N10}

\begin{document}

\begin{abstract}
We consider the existence of periodic traveling waves in a bidirectional Whitham
equation, combining the full two-way dispersion relation from the incompressible Euler equations
with a canonical shallow water nonlinearity.  Of particular interest is the existence 
of a highest, cusped, traveling wave solution, which we obtain as a limiting case at the end of 
the main bifurcation branch of $2\pi$-periodic traveling wave solutions continuing from the zero state.  
Unlike the unidirectional
Whitham equation, containing only one branch of the full Euler dispersion relation, where
such a highest wave behaves like $|x|^{1/2}$ near its crest, the cusped
waves obtained here behave like $|x\log|x||$. Although the linear operator involved in this equation can be easily represented in terms of an integral operator, it maps continuous functions out of the H\"older and Lipschitz scales of function spaces by introducing logarithmic singularities. Since the nonlinearity is also of higher order than that of the unidirectional Whitham equation, several parts of our proofs and results deviate from those of the corresponding unidirectional equation, with the analysis of the logarithmic singularity being the most subtle component. This paper is part of a longer research programme for understanding the interplay between nonlinearities and dispersion in the formation of large-amplitude waves and their singularities.
\end{abstract}

\maketitle

\section{Introduction}

Given the great complexity of the Euler equations, which fundamentally describe the flow of an incompressible, inviscid fluid over an impenetrable bottom,
it has long been considered advantageous to find simpler models that approximate the dynamics of the free surface in particular asymptotic regimes.  
Arguably, the most famous such approximation is the
well-studied KdV equation which, in dimensional form, can be written as
\begin{equation}\label{e:kdv}
u_t+\sqrt{gh_0}\left(1+\frac{1}{6}h_0^2\partial_x^2\right)u_x+\frac{3}{2}\sqrt{\frac{g}{h_0}}uu_x=0,
\end{equation}
where here $x$ denotes the spatial variable, $t$ denotes the temporal variable, and $u=u(x,t)$ is a real-valued function describing the fluid surface; further,
$g$ denotes the constant due to gravitational acceleration and $h_0$ denotes the undisturbed fluid depth.
The KdV equation is well known to describe unidirectional propagation of small amplitude, long wave phenomena in a channel of water, most notably periodic and solitary waves,
but is also known to lose relevance for short and intermediate wavelength regimes where wave features
such as wave breaking and surface singularities may be observed in other equations.

Heuristically, this may be explained as follows: if one linearizes the KdV equation
about the flat state $u {=} 0$, a straightforward calculation implies that this linearized equation admits plane wave solutions of the form $e^{ik(x-ct)}$ provided that
\[
c_{\rm KdV}(k)=\sqrt{g h_0}\left(1-\frac{1}{6}(kh_0)^2\right),
\]
while performing the analogous calculations on the Euler equations one finds the linearized system admits such plane wave solutions provided that \(c_{\rm Euler}^2(k)=\frac{g\tanh(kh_0)}{k}\), that is
\begin{equation}\label{eq:Euler_dispersion}
c_{\rm Euler,\pm}(k)= \pm\sqrt{\frac{g\tanh(kh_0)}{k}}.
\end{equation}
Note that while the KdV has only one phase speed, the Euler equations has \emph{two} branches of the phase speed $c$.  This is a reflection of the
fact that the Euler equation generically supports \emph{bidirectional} propagation of waves, 
while the KdV equations is derived under the assumption of \emph{unidirectional wave propagation}.
Concentrating on the positive branch of $c_{\rm Euler}(k)$, the connection between these two phase speeds 
is given through the expansion
\[
c_{\rm Euler,+}(k)=c_{\rm KdV}(k)+\mathcal{O}(|kh_0|^4),
\]
so that the KdV equation can be seen to approximate to second order
the positive branch of the Euler phase speed in the long-wave regime $|kh_0|\ll 1$.  In fact, solutions
of \eqref{e:kdv} are known to exist and converge to those of the water wave problem at the order
of $\mathcal{O}(h_0^2k^2)$ during an appropriate time interval; see \cite[Section 7.4.5]{LannesBook} for details.
Outside of this regime, however, it is clear that the KdV phase speed is a poor approximation of 
that for the Euler equations and hence 
the KdV equation should not be expected to describe short, or even intermediate, wavelength phenomena.

To better describe short wave phenomena, Whitham suggested to replace the linear phase speed in the KdV equation with the \emph{exact, unidirectional} phase speed from
the Euler equations.  This leads to the nonlocal evolution equation
\begin{equation}\label{e:w}
u_t+\mathcal{M}u_x+\frac{3}{2}\sqrt{\frac{g}{h_0}}uu_x=0
\end{equation}
where here $\mathcal{M}$ is a Fourier multiplier defined by its symbol via
\[
\widehat{Mf}(k)=\sqrt{\frac{g\tanh(h_0k)}{k}}~\hat{f}(k).
\]
Denoting $D=\frac{1}{i}\partial_x$, we can thus formally write $\mathcal{M}=\sqrt{\frac{g\tanh(h_0D)}{D}}$.  
As \eqref{e:w} combines the full unidirectional phase speed from the Euler equations with the canonical shallow water nonlinearity, Whitham advocated that \eqref{e:w}, now referred to as
the \emph{Whitham equation}, should admit periodic and solitary waves while at the same time
allowing for wave breaking and surface singularities.  Much recent activity has verified these 
claims for the Whitham equation: equation \eqref{e:w} admits both solitary \cite{EGW12} and periodic
\cite{EK09,EK11} waves, but also features wave breaking \cite{CE98,MR3682673} and a highest cusped
traveling wave solution \cite{EWhighest}.  Notably, the cusped solutions in \cite{EWhighest} were shown to exist
through a global bifurcation argument, continuing off a local branch of small amplitude periodic
traveling waves bifurcating from the zero state, and were shown to be smooth away from their highest point (the crest) and
behave like $|x|^{1/2}$ near the crest.  It should also be noted that the Whitham equation
\eqref{e:w} features the same kind of Benjamin--Feir instability \cite{HJ15,SKCK14} as the
Euler equations; see also \cite{HJ15_2} where additional effects of constant vorticity and surface tension
were considered. We refer to \cite{MR3603270} for a study on the symmetry and decay properties of solitary wave solutions
of \eqref{e:w}, as well as \cite{Arn16,EEP15} where the associated Cauchy problem is studied.

Despite the success of the Whitham equation \eqref{e:w}, there are still water wave phenomena
that it does not capture.  For example, it is known that the Euler equations admits high-frequency
(non-modulational) instabilities of small amplitude periodic traveling waves: see \cite{RS86,MR3640555}
and references therein.  In \cite{MR3640555}, however, it was shown that the unidirectional nature of 
the Whitham equation prohibits such instabilities from manifesting.  Indeed, there it was 
argued that the bidirectionality of the Euler equations was the key underlying feature allowing
for the possibility of such instabilities.  Furthermore, the irrotational Euler equations are known to
admit peaked waves, i.e., traveling wave solutions with bounded but discontinuous derivatives at their highest point, 
with a corner at each crest with an interior angle of 120$^o$~\cite{AFT82}.  The Whitham equation \eqref{e:w} instead features cusped waves, having exactly half\footnote{Indeed, the wave constructed in \cite{EWhighest} was shown
to have optimal global regularity $C^{1/2}(\RM)$.} the regularity of the highest waves of the Euler equations \cite{EWhighest}. In light of \eqref{eq:Euler_dispersion} it is tempting to expect that this is not due to some bad modeling aspect of the Whitham equation, but to its unidirectionality. \emph{The goal of this paper is to analyze the steady periodic waves of the corresponding bidirectional Whitham equation, and to see how this influences the existence and features of a possible highest wave for such an equation.} In particular, is it peaked as for the Euler equations? Answering  that question is part of a longer research programme for understanding the interplay between nonlinearities and (nonlocal) dispersion in the formation of large-amplitude waves and their singularities.

In this paper, we consider the following full-dispersion shallow water wave model, given here in 
dimensional form:
\begin{equation}\label{dim_biw}
\left\{\begin{aligned}
\eta_t&=-\frac{1}{\sqrt{gh_0}}\mathcal{K}u_x-\sqrt{\frac{g}{h_0}}\left(\eta u\right)_x\\
u_t&=-\sqrt{\frac{g}{h_0}}\left(\eta_x+uu_x\right),
\end{aligned}\right.
\end{equation}
where the operator $\mathcal{K}$ is a Fourier multiplier defined by its symbol via
\[
\widehat{\mathcal{K}f}(k)=\frac{g\tanh(kh_0)}{k}~\hat{f}(k),
\]
that is, \(\mathcal{K}:=\frac{g\tanh(h_0D)}{D}\), where as above $D=\frac{1}{i}\partial_x$.
Here, $\eta$ represents the free surface, and $u=\varphi_x$ where $\varphi(x,t)=\varphi(x,\eta(x,t),t)$ is the trace of the velocity potential at the surface interface. The dispersion relation for \eqref{dim_biw} agrees exactly 
with that of the full Euler equation, so that this is a bidirectional
equation with two branches of the linear phase speed given in \eqref{eq:Euler_dispersion}. The model \eqref{dim_biw} also appeared in \cite[Section 6.3]{MR3668593}, where it was
described as a linearly well-posed regularization of the linearly ill-posed, yet completely integrable, Kaup system
\[
\left\{\begin{aligned}
&\eta_t=-\sqrt{gh_0}\left(1+\frac{1}{3}h_0^2\partial_x^2\right)u_x	-\sqrt{\frac{g}{h_0}}(\eta u)_x=0\\
&u_t=-\sqrt{\frac{g}{h_0}}\left(\eta_x+uu_x\right)=0.
\end{aligned}\right.
\]
Indeed, one can see that the phase speed for the Kaup system agrees to $\mathcal{O}(|kh_0|^2)$ with
that of \eqref{dim_biw}.
Other, and more involved,  full-dispersion equations, are given in \cite{LannesBook} and  \cite{MR3026551}.  The system \eqref{dim_biw} can be derived as an ad-hoc bidirectionalization of the Whitham equation or, as in \cite{MR2991247} and \cite{MKD2015}, via a formal expansion of the Dirichlet--Neumann operator appearing in the free-surface water-wave equations. Although the analytical existence theory for this equation is conditional \cite{EPW17,klein2017whitham}, it displays nice \emph{qualitative} properties for solutions with strictly positive surface elevation: it significantly outperforms the KdV equation in experimental settings \cite{Carter2017,MR3523508}, and the steady periodic waves of supercritical wave speed are stable in the appropriate regimes \cite{CJ17}. A proof of conditionally stable solitary waves in the spirit of \cite{EGW12} is in preparation, too \cite{NW18}.

As we will see below, small amplitude periodic traveling wave solutions of \eqref{dim_biw}  can be 
shown, at particular wave speeds, 
to bifurcate from the trivial solution $(\eta,u)=(0,0)$ through the use of 
elementary bifurcation theory and a Lyapunov--Schmidt reduction.  By numerically continuing this local
branch of solutions, we observe that the waves approach a highest wave, which at lower resolutions does indeed seem to be peaked. From a more detailed analysis, however, it appears that unlike the full Euler equations the highest wave of \eqref{dim_biw} is still cusped at its highest point. To understand this more rigorously, we justify the numerical observations through
the use of a global bifurcation argument in the spirit of \cite{BTbook,EK11}.  By combining this
global argument with {a priori} estimates on a wave of extreme height we establish a highest, cusped, almost everywhere smooth, traveling wave solution of \eqref{dim_biw}, which behaves 
as $|x\log|x||$ near the crest. The introduction of bidirectionality therefore has a  twofold effect on the highest wave: it increases (doubles) its regularity, but it also introduces a logarithmic factor such that the derivative is not any more bounded but blows up logarithmically. The latter can be explained by the functional analysis of Fourier multiplier operators of integer order, see Subsection~\ref{subsec:functional}.

We remark that the present paper is an extension of the recent work
\cite{EWhighest}, where the authors performed an analogous study on the unidirectional
Whitham equation. The integral kernel associated
with the Fourier multiplier $\mathcal{K}$, however, introduces novel difficulties in the analysis coming from its logarithmic blow-up at low frequencies;  see Lemma \ref{p:kernel}(iii) below.  As a consequence,
it is not possible to capture the global regularity of the highest wave in terms of classical
H\"older, or even H\"older--Zygmund, spaces.  We believe our analysis sheds some light
on a more general existence theory of extreme waves associated to dispersive nonlocal equations, and this 
is planned to be reported on in the future.

The outline for our investigation is as follows.  In Section \ref{s:prelim} we lay out  the analytic
preliminaries.  Most importantly, we perform a detailed study
of the integral kernel $K$ associated to the Fourier multiplier $\mathcal{K}$ above, together
with its $2\pi$-periodic periodization.  Due to the fact that the integral kernel associated
to $\mathcal{K}$ is known in closed form
we are able to easily describe the singular nature of the kernel $K$ near zero-frequency, together
with its monotonicity properties.  Mark that the corresponding analysis in \cite{EWhighest} 
for the unidirectional Whitham equation was significantly more complicated, due to the fact 
that integral kernel associated to $\mathcal{M}$ is not known in such a clean form\footnote{Indeed,
in \cite{EWhighest}, using the theory of completely monotone functions, 
the authors provide the first closed form expression for the integral kernel associated to $\mathcal{M}$,
although their expression is not as explicit as that for the integral kernel associated to \(\mathcal{K}\).}.  In Section \ref{s:num} we report on a numerical
investigation of the global bifurcation diagram, continuing from the zero state, for the profile equation associated to \eqref{dim_biw}.
The numerics are then used to motivate the analytical theory in the remainder of the paper.  
In Section \ref{sec:a priori} we prove some a priori estimates and lemmas concerning 
periodic traveling wave solutions of \eqref{dim_biw} of maximum height.  Finally, the local and
global bifurcation analysis for our solutions is performed in Section \ref{sec:bifurcation}, where it is shown
that there is a sequence of waves converging to a logarithmically cusped wave of greatest height, thus establishing
our main result Theorem \ref{thm:main}.  


\section{Preliminaries}\label{s:prelim}
We consider $2\pi$-periodic solutions of the full-dispersion, bidirectional shallow water system \eqref{dim_biw}.
In non-dimensional form, they read as
\begin{equation}\label{biw}
\left\{\begin{aligned}
\eta_t&=-\mathcal{K}u_x-(\eta u)_x\\
u_t&=-\eta_x-uu_x
\end{aligned}\right.
\end{equation}
and, with slight abuse of notation, the operator $\mathcal{K}$ is now a Fourier multiplier defined by its symbol via
\begin{equation}\label{symbol}
\widehat{\mathcal{K}f}(k)=\frac{\tanh(k)}{k}~\hat{f}(k), \quad\text{i.e.}\quad \mathcal{K}:=\frac{\tanh(D)}{D}.
\end{equation}
Precisely, \eqref{biw} is obtained from \eqref{dim_biw} via the rescaling
\[
t\mapsto\sqrt{\frac{h_0}{g}}~t,\quad x\mapsto h_0~x.
\]
Our primary concern is the existence of traveling wave solutions of \eqref{biw}, which are 
solutions of the form $(\psi,\varphi)(x,t)=(\eta,u)(x-ct)$ where, again abusing notation, the profiles $\varphi$ and $\psi$ satisfy
the nonlocal system
\begin{align*}
&\mathcal{K}\varphi_x+\left(\psi ({\varphi}-c)\right)_x=0\\
&\psi_x=c\varphi_x-\varphi \varphi_x
\end{align*}
Integrating both equations, one sees that localized solutions must satisfy the scalar profile equation
\begin{equation}\label{eq:profile'}
\mathcal{K} \varphi=\varphi\left(c-\frac{1}{2} \varphi\right)(c-\varphi).
\end{equation} 
More generally, the Galilean transformation \((\varphi,c) \mapsto (\varphi,c)  + (\lambda,\lambda)\) can be used to eliminate one of the constants of integration, although not both. We remark that a theory for a class of general nonlinearities is planned in a future investigation.

\begin{figure}[t]
\begin{center}
\includegraphics[scale=0.45]{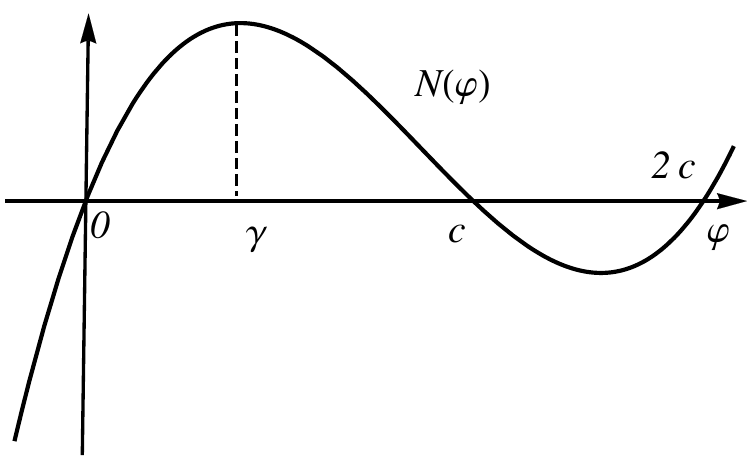}
\caption{\small A plot of the cubic function $N(\varphi)$ for a given $c>0$.  Of key importance is the locally quadratic
nature of $N$ near its first positive critical point $\gamma$.
}
\label{F:Npic}
\end{center}
\end{figure}

Factoring the third-order polynomial in \eqref{eq:profile'} near its critical point at $\varphi=\gamma$, our equation reads
\begin{equation}\label{profile}
\mathcal{K} \varphi = N(\varphi)
\end{equation}
with
\begin{equation}\label{eq:N_doubleroot}
N(\varphi) = N(\gamma) + \textstyle{\frac{1}{2}}(\varphi - \gamma - \sqrt{3}c) \left(\varphi - \gamma \right)^2, 
\end{equation}
defined by the right-hand side of \eqref{eq:profile'} and 
\[
\gamma := c \left(1 - \textstyle{\frac{1}{\sqrt{3}}} \right)
\] 
being the smallest root of \(N^\prime(\varphi) = 0\); see Figure \ref{F:Npic}. In particular, observe
that  \(N(\varphi)\) has no critical points for \(\varphi < \gamma\), 
and is strictly increasing on the same half-line. It increases through the origin to its local maximum at \(\varphi = \gamma\).  As the forthcoming analysis will show, 
the number $\gamma$ will be the maximum of the highest wave to be, and the 
quadratic nature of $N$ near $\varphi=\gamma$ will cause the resulting singularity at the peak.

We devote the remainder of this section to establishing key properties of the Fourier multiplier $\mathcal{K}$
as well as to set up the functional framework used in our analysis.

\subsection{The integral operator \(\mathcal{K}\)}
To make sense of the operator $\mathcal{K}$, we utilize the Fourier transform.  Throughout,
the operator $\mathcal{F}$ will denote the extension to the space of tempered distributions $\mathcal{S}'(\mathbb{R})$
of the Fourier transform
\[
\mathcal{F}(f)(\xi):=\int_\mathbb{R} f(x)\exp(-i\xi x)\, dx
\]
on the Schwartz space $\mathcal{S}(\mathbb{R})$, with inverse 
$ \mathcal{F}^{-1}(f)(\xi) =\frac{1}{2\pi}\mathcal{F}(f)(-\xi)$.
Observe that with this normalization, $\mathcal{F}$ defines a unitary operator on $L^2(\mathbb{R};\mathbb{C})$.
The operator $\mathcal{K}$ on $\mathcal{S}(\mathbb{R})$ may then be understood via the inverse Fourier transform
representation
\[
\mathcal{K}f(x) = \mathcal{F}^{-1}\left(\frac{\tanh(\xi)}{\xi}\hat{f}(\xi)\right)(x).
\]
By the convolution theorem, one can equivalently introduce the integral kernel
\begin{equation}\label{K}
K(x)=\mathcal{F}^{-1}\left(\frac{\tanh(\xi)}{\xi}\right)(x)
\end{equation}
and define the action of $\mathcal{K}$ on $\mathcal{S}(\mathbb{R})$ by convolution with $K$, 
that is,
\[
\mathcal{K}f(x)=K\ast f(x) = \int_\mathbb{R}K(x-y)f(y) \, dy.
\]
By duality, this action can be extended to any $f \in \mathcal{S}^\prime(\RM)$. 

In the forthcoming analysis, we will utilize several positivity, monotonicity, and asymptotic
properties of the kernel $K$.  To aid in this description, we make the following definition.

\begin{definition}
Let $0\leq a<b\leq\infty$.  A function 
$g:(a,b)\to\mathbb{R}$ is  called {\em completely monotone} if it is of class $C^\infty$ and
\begin{equation}
\label{eq:completely monotone}
(-1)^n g^{(n)}(\lambda)\ge 0
\end{equation}
for all \(n \in \ZM_0\) and all 
$\lambda\in(a,b)$.
\end{definition}

A proof of that a general class of kernels, including \(K\), are completely monotone on $(0,\infty)$ can be found in \cite{EWhighest}, and is due to E. Wahl\'en. In our case, \(K\) is explicitly known, and the complete monotonicity follows directly. 
Before stating this result, we make the following convention. Given any real-valued functions $f$ and $g$, we say that $f\lesssim g$, or $f(x) \lesssim g(x)$, if there exists a constant $C>0$ such that $f(x)\leq C g(x)$ for all $x$ in the domain of interest. If no specific domain is indicated, the statement is understood to be globally valid.  The opposite relation $\gtrsim$ is defined analogously, and we write $f\eqsim g$ when \(f \lesssim g \lesssim f\).
In any  chain of inequalities, we will also feel free to denote by $C$ harmless constants with possibly different values.

\begin{lemma}\label{p:kernel}
The integral kernel  $K$ is given explicitly by
\[
K(x)=2\log\left|\coth(\frac{\pi x}{4})\right|,
\]
and is completely monotone on $(0,\infty)$. In particular, 
\begin{itemize}
\item[(i)] \(K\) is real-valued, even, strictly positive on $\mathbb{R}\setminus\{0\}$, and satisfies
\[
\|K\|_{L^1(\mathbb{R})}=\mathcal{F}\left(K\right)(0)=1.\\[6pt]
\] 
\item[(ii)] \(K \in C^\infty(\RM \setminus \{0\})\), and for any 
$s_0\in(0,\pi/2)$
and \(n \geq 0\), one has
\[
|\partial_{x}^n K(x)| \lesssim  \exp(-s_0 |x|), 
\]
uniformly for all \(x \geq 1\).\\[-6pt]
\item[(iii)]  \(K\) has the canonical representation
\begin{equation}\label{eq:K_twoparts}
K(x) = -2\log\left|\frac{\pi x}{4}\right|+ K_\text{reg}(x), 
\end{equation}
with  \(K_\text{reg} \in C^\infty(\RM)\) being the regular part of \(K\). 
As \(x \to 0\), one has the asymptotic expansion  
$K_{\rm reg}(x)=\frac{1}{24}\pi^2x^2+\mathcal{O}(x^4)$,
which is valid under term-wise differentiation. 
\end{itemize}
\end{lemma}

\begin{proof}
The explicit formula for the Fourier transform can be found, for instance, in \cite[Section 5.5.4]{Bhatia}\footnote{In
\cite{Bhatia}, the author provides formulas only up to multiplicative constants.  In this case, the constant
can be found by enforcing the requirement that $\int_{\mathbb{R}}K(x)dx=1$.}
or \cite[Section 1.7]{Oberhettinger}. 
It immediately follows that \(K\) is completely monotone on $(0,\infty)$, and specifically 
the properties given in (i) are immediate. Concerning (ii), the function 
\(\xi \mapsto \frac{\tanh \xi}{\xi}\) is analytic in the strip 
\(\RM \times (-\pi/2,\pi/2)\)
in the complex plane. 
By shifting contours and appealing to Cauchy's integral theorem, it can then be shown that
\[
\int_{-\infty}^\infty\frac{\tanh(x)}{x}e^{i\xi x}dx=
\int_{-\infty}^\infty\frac{\tanh(x+is_0)}{x+is_0}e^{i\xi (x+is_0)}dx
\]
for all $s_0\in(0,\pi/2)$, from which exponential decay follows from the integrability
of $\textstyle{\frac{\tanh(x+is_0)}{x+is_0}e^{i\xi x}}$; see \cite{EWhighest} for details
in a closely related context.

To prove (iii), because \(K\) is smooth outside of the origin it is enough to establish the representation \eqref{eq:K_twoparts} for \(|x| \ll 1\). But there it follows from the analytic expansions of \(\cosh(x)\) and \(\sinh(x)\) when combined with rudimentary properties of the logarithm. The asymptotic formula for \(K_\text{reg}\) is obtained via the same expansion.
\end{proof}

\subsection{The operator \(\mathcal{K}\) on periodic functions}

As our interest lies in periodic solutions of \eqref{profile}, we now describe how $\mathcal{K}$ acts on periodic
functions.  Since $K$ lies in $L^1(\mathbb{R})$, given any $f\in L^\infty(\mathbb{R})$ that is $2\pi$-periodic, we can
write
\[
\mathcal{K}f(x)=\int_{-\pi}^\pi \left(\sum_{k\in\mathbb{Z}}K(x-y+2\pi k)\right)f(y)dy=: \int_{-\pi}^{\pi}  K_p(x-y)f(y)dy.
\]
By Lemma~\ref{p:kernel}, the periodized kernel $K_p$ is readily seen to converge {absolutely}  and to admit the Fourier series expansion
\[
K_p(x)=\sum_{n\in\mathbb{Z}}\frac{\tanh(n)}{n}\exp(inx).
\]
Consequently, the {convolution theorem} guarantees that $\mathcal{K}$ acts on 
$2\pi$-periodic functions in the same way as on functions on the line, namely
as
\[
\mathcal{K}f(x)=\sum_{n\in\mathbb{N}}\hat{f}(n)\left(\frac{\tanh(n)}{n}\right)\exp(inx),
\]
for any \(2\pi\)-periodic function or tempered distribution $f$.  

\begin{lemma}\label{lemma:Kp}
The periodic integral kernel $K_p$ is completely monotone on the half-period \((0,\pi)\). Moreover:
\begin{itemize}
\item[(i)] \(K_p\) is even, strictly positive on $\mathbb{R}\setminus 2\pi\mathbb{Z}$, and satisfies
\[
\|K{_p}\|_{L^1(-\pi,\pi)}=1.\\[6pt]
\] 

\item[(ii)] \(K_p\) is smooth on \(\RM \setminus 2\pi\mathbb{Z}\).\\[-6pt]

\item[(iii)]  \(K_p\) has the canonical representation
\begin{equation}\label{eq:Kp_twoparts}
K_p(x)=-2\log\left|\frac{\pi x}{4}\right|+K_{p,{\rm reg}}(x),
\end{equation}
with  \(K_{p,\text{reg}} \in C^\infty(-\pi,\pi)\) being the regular part of \(K_p\).
\end{itemize}
\end{lemma}

\begin{remark}\label{rem:sign-changing}
It follows from the strict positivity of \(K_p\) that any periodic solution of \eqref{profile} that attains \(\varphi(x_0) = 0\) for some \(x_0\) is either trivially zero, or must change signs.
\end{remark}

{
\begin{remark}
The fact that $K_p$ is completely monotone on $(0,\pi)$ follows immediately from \cite[Proposition 3.2]{EWhighest},
where the authors used Bernstein's theorem to show that the periodization of an even, integrable function on $\RM$ that is completely
monotone on $(0,\infty)$ is itself completely monotone on a half-period.  Nevertheless, here we provide
a more direct proof relying simply on the complete monotonicity of $K$ and the decay of it and its
derivatives.
\end{remark}
}

\begin{proof}
We observe first that the periodized kernel \(K_p\) inherits its parity and positivity directly from the kernel 
\(K\) studied in the previous section.  
Recall {also} that, with the origin as the sole exception, the kernel \(K\) and 
all its derivatives are smooth with exponential decay. For \(K_p\) as in the proposition we then have 
for any fixed $x\in(0,\pi)$
\begin{equation}\label{eq:monotone_periodic}
\begin{aligned}
(-1)^n\partial_x^n K_p(x) &= (-1)^n\sum_{k \in \ZM} K^{(n)}(x + 2k\pi)\\ 
&= (-1)^n\sum_{k \geq 0} \left( K^{(n)}(x+2k\pi) + K^{(n)}(x-2(k+1)\pi)\right).
\end{aligned}
\end{equation}
If \(n=2m\) is even, it is clear from the positivity of \(K^{(2m)}\),which is even, that \(K_p^{(2m)}\) is positive as well. So assume that \(n=2m +1\) is odd, and let \(a_k = x + 2k\pi\),  \(b_k = x - {2(k+1)\pi}\). Then \(K^{(2m+1)}(a_k) < 0\), whereas \(K^{(2m+1)}(b_k) > 0\), for all \(x \in (0,\pi)\) and all integers \(k \geq 0\), by the complete monotonicity of \(K\). We thus want 
\[
|K^{(2m+1)}(a_k)| >  |K^{(2m+1)}(b_k)|. 
\]
By the evenness of \(K^{(2m)}\), we have \(|K^{(2m+1)}(\zeta)| = |K^{(2m+1)}(-\zeta)|\) for any \(\zeta \neq 0\). And since \(K^{(2m+2)}\) is positive, \(|\zeta| \mapsto |K^{{(2m+1)}}(|\zeta|)|\) is furthermore a strictly decreasing function of \(|\zeta|\), so that 
\[
|a_k| < (k+1/2)2\pi < |b_k|
\] 
guarantees that \(|K^{(2m+1)}(a_k)| >  |K^{(2m+1)}(b_k)|\). Hence, the sum in \eqref{eq:monotone_periodic} is strictly positive for all \(x \in (0, \pi)\), thus verifying complete monotonicity on the half-period $(0,\infty)$.

Properties (i) and (ii) follow from the evenness, positivity, and decay of $K$ established
in Lemma \ref{p:kernel} above. Finally,
\begin{align*}
K_p(x) &= K(x) + \sum_{k \neq 0} K(x + 2k\pi)\\ 
&= -\log(\pi x) + K_\text{reg}(x) + \sum_{k \neq 0} K(x + 2k\pi),
\end{align*}
for \(x \in (0,\pi)\), which gives the representation formula for \(K_p\). 
\end{proof}

\subsection{Functional-analytic framework: H\"older and Zygmund spaces }\label{subsec:functional}
Before we address the existence of solutions of \eqref{profile}, we describe the functional-analytic framework
used throughout this work.  
In principle, we wish to work on a space of functions capable of capturing an appropriate scale of
smoothness, while at the 
same time behaving well under the action of Fourier multipliers.  
It turns out that such a space is given by the H\"older (more precisely, Zygmund) spaces, which we now briefly describe.

For $0<\alpha<1$ we define the space $C^\alpha(\mathbb{S})$ of $\alpha$-H\"older continuous functions on {the unit circle} $\mathbb{S}$
to consist of all continuous, $2\pi$-periodic functions $u$ such that 
\[
 |u(x)-u(y)|   \lesssim  |x-y|^\alpha
\]
for all $x,y\in\mathbb{R}$, and we equip $C^\alpha(\mathbb{S})$ with the norm
\[
\|u\|_{C^\alpha(\mathbb{S})}:=\sup_{x\neq y}\frac{|u(x)-u(y)|}{|x-y|^\alpha}, \qquad \alpha\in(0,1).
\]
For $k=0,1,2,\ldots$ we take $C^k(\mathbb{S})$ to denote all $k$-times continuously differentiable functions
on $\mathbb{S}$, equipped with the norm 
\[
\|u\|_{C^k(\mathbb{S})}:=\sum_{j=0}^k\left\|\partial_x^j u\right\|_{L^\infty(\mathbb{S})}, \qquad k=0,1,2,\ldots.
\]
{If then} $s=k+\alpha$ for some $k=0,1,2,\ldots$ and $\alpha\in(0,1)$ we define $C^s(\mathbb{R})$ to be the set of all
functions $u\in C^k(\mathbb{S})$ such that $\partial_x^ku\in C^\alpha(\mathbb{S})$, and we equip this space with the norm
\[
\|u\|_{C^s(\mathbb{S})}:=\|u\|_{C^k(\mathbb{S})}+\|u\|_{C^\alpha(\mathbb{S})}. 
\]

While the H\"older spaces provide a quantitative measurement of the modulus of continuity of a function, it is not immediately clear
how such spaces behave under the action of Fourier multipliers.
Thankfully, H\"older spaces have a particularly nice characterization, similar to that of the Lebesgue or Sobolev spaces, in terms
of Littlewood--Paley theory.  Indeed, if we consider the partition of unity
\[
1=\sum_{j=0}^\infty\psi_j(n)^2
\]
with $\psi_j$ supported on $2^j\leq |n|< 2^{j+1}$ and $\psi_j(n)=\psi_1(2^{1-j}n)$ for $j\geq 1$, and on $|n|\leq 2$
when $j=0$,  then Littlewood-Paley theory
gives the following.

\begin{proposition}{\cite{Taylor}}
If $u\in C^s(\mathbb{S})$ for some $s\geq 0$, then
\begin{equation}\label{zygmund}
\sup_{j} 2^{js}\left\|\psi_j^2(D)u\right\|_{L^\infty(\mathbb{S})}<\infty.
\end{equation}
Furthermore, \eqref{zygmund} guarantees $u\in C^s(\mathbb{S})$ for non-integer values of $s$.
\end{proposition}

It follows that as long as $s$ is not an integer, the H\"older spaces can be characterized completely by Fourier series, and hence
behave nicely under the action of general Fourier multipliers (for more details on this and other statements in section, we refer the reader to \cite[Chap. 17]{Taylor}).  To state this precisely, introduce 
periodic Zygmund spaces $C^s_*(\mathbb{S})$, $s\geq 0$, 
consisting of all continuous functions $u$ on $\mathbb{S}$ such that  \eqref{zygmund} holds.  For each $s\geq 0$
we equip $C^s_*(\mathbb{S})$ with the obvious norm
\[
\|u\|_{C^s_*(\mathbb{S})}:=\sup_{j} 2^{js}\left\|\psi_j^2(D)u\right\|_{L^\infty(\mathbb{S})}
\]
and note that, under this norm, the Zygmund spaces are Banach spaces and that
\[
C^s(\mathbb{S})=C^s_*(\mathbb{S})~~{\rm if}~~s\in(0,\infty)\setminus\mathbb{N},
	\qquad   C^k(\mathbb{S})  \hookrightarrow  C^k_*(\mathbb{S})~~\textrm{if}~~s\in\NM_0.
\]
The next result asserts that Fourier multipliers act on the Zygmund spaces in much the same way as they act on Sobolev spaces.

\begin{proposition}{\cite{Taylor}}
Suppose that $f:\mathbb{R}\to\mathbb{R}$ is a smooth function such that, for some $m\in\mathbb{R}$ and any \(k \in \NM_0\), 
\[
\left|\partial_\xi^k f(\xi)\right|\lesssim \left(1+|\xi|\right)^{m-k}
\]
for all $\xi \in\mathbb{R}$.  Then 
\(
f(D)\in\mathcal{L}\left(C^{s+m}_*(\mathbb{S}),C^{s}_*(\mathbb{S})\right)
\)
for all \(s \geq 0\).
\end{proposition}

Now, if $f\in C^\alpha(\mathbb{S})$ with $\alpha>1/2$ one has that 
\[
f(x) \equiv \sum_{k\in\mathbb{Z}}\hat{f}(k) \exp(ikx) \quad\textrm{with}\quad\sum_{k\in\mathbb{Z}}|\hat{f}(k)|<\infty;
\]
in particular, the Fourier series of $f$ converges absolutely for all $x\in\mathbb{R}$.
Since \(\xi \mapsto \tanh(\xi)/\xi\) is smooth with
\[
 \left|\partial_\xi^k\left(\frac{\tanh(\xi)}{\xi}\right)\right| \lesssim \left(1+|\xi|\right)^{-1-k}
\]
for all \(k\in\NM_0\), it follows that 
\begin{equation}\label{eq:mapping}
\mathcal{K}: C^s_*(\mathbb{S})\to C^{s+1}_*(\mathbb{S})\quad  \text{for all } s\geq 0.
\end{equation}
The fact that $\mathcal{K}$ has a negative \emph{integer} order presents an interesting challenge
in the forthcoming analysis, especially in {regard} to the global regularity of solutions
of \eqref{profile}. Indeed, observe that if $f\in C^0(\mathbb{S})\subset C^0_*(\mathbb{S})$
then the function $\mathcal{K}f$ is only guaranteed to belong to $C^1_*(\mathbb{S})$, and hence
may not be continuously differentiable or even Lipschitz continuous on $\mathbb{S}$.  We will
return to this point at the end of Section \ref{sec:a priori}.

\section{Numerical Observations}\label{s:num}
Our aim is to establish the existence of a cusped traveling wave solution
of the nonlocal profile equation \eqref{profile}.  Such a solution will be shown to exist through
the construction of a global bifurcation curve of traveling wave solutions
with fixed period, the end of which will be a {logarithmically} cusped wave.  A similar analysis was recently 
performed on the unidirectional Whitham equation \eqref{e:w}.  {That series of papers} started as a theoretical and numerical investigation of the local \cite{EK09} and global \cite{EK11} bifurcation of traveling wave solutions, and ended with the establishing of a highest, cusped, wave in the recent investigation \cite{EWhighest}. Just as the numerical investigations in \cite{EK09,EK11} served as a starting
point for the rigorous search for cusped solutions of \eqref{e:w}, the purpose of this section is to provide analogous numerical findings 
for the equation \eqref{profile}. We thereby hope to motivate the analytical theory by pointing out some key features of solutions along the global bifurcation branch that will be central to our later analysis, as well as some that are conjectures that have not yet been proved analytically.
See also \cite{CJ17}, where the authors numerically investigate the global bifurcation and dynamic stability of periodic traveling wave solutions
of various bidirectional, full-dispersion water wave models.

\begin{figure}[t]
\begin{center}
\includegraphics[scale=0.5]{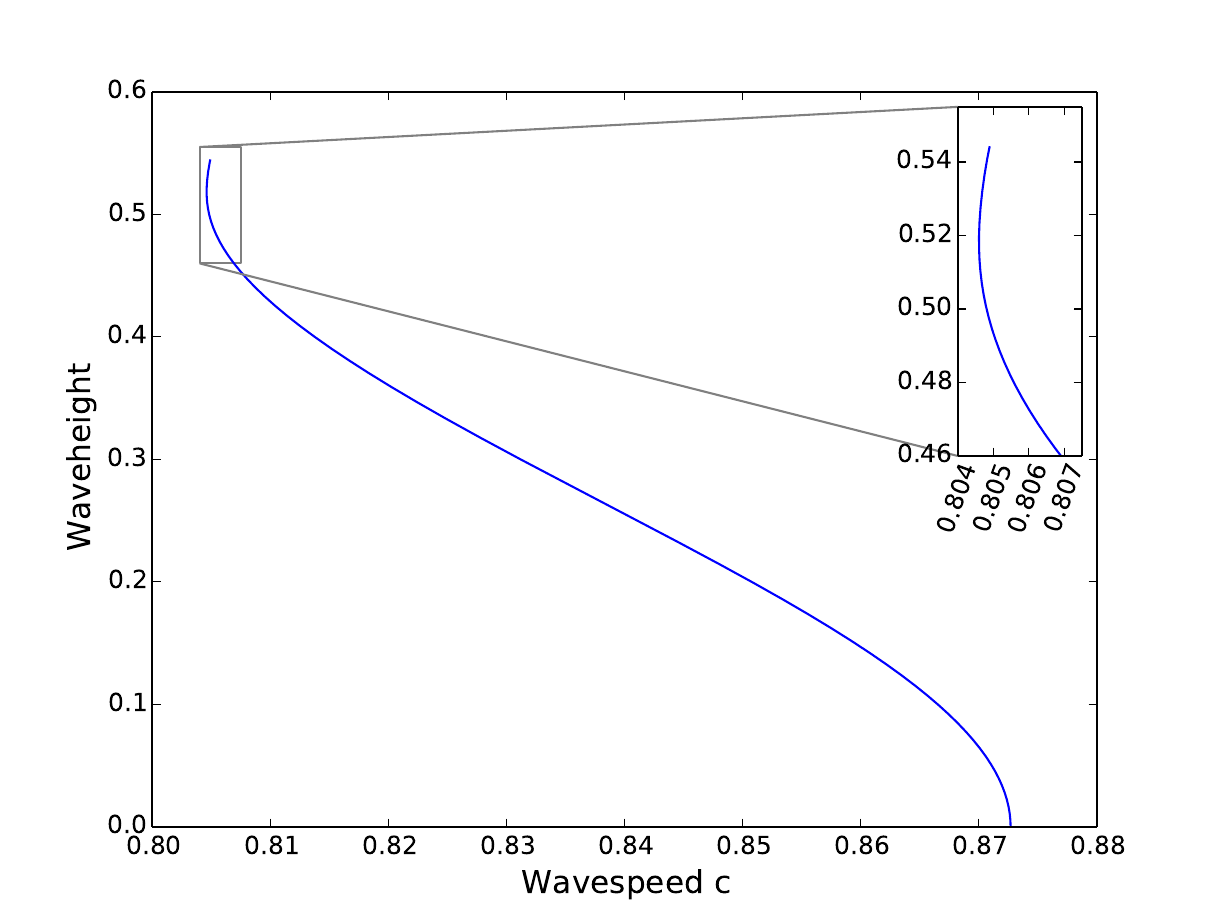}
\end{center}
\caption{ \small A numerical approximation of the bifurcation branch of even, {one-sided monotone,} $2\pi$-periodic
solutions of \eqref{profile} with wave speed $c$.  A zoom is provided near the turning
point. }
\label{fig:bifcurve1}
\end{figure}

In Proposition \ref{prop:local} and Theorem \ref{thm:global} we prove the existence of a local and global branch, respectively, 
of small-amplitude, $2\pi$-periodic even and {one-sided monotone solutions $\varphi({\tau})$} of \eqref{profile}
with wave speed $c({	\tau})$. These bifurcate from the trivial solution when $c_0^2 =  \tanh{(1)}$, and may be numerically
approximated through the use of a spectral cosine collocation method which will be outlined in Appendix
\ref{Appendix}.  By the discussion at the beginning of Section~\ref{s:prelim}, we expect
this bifurcation curve to continue to a highest wave having maximum height at $\varphi=\gamma$.
Here, we briefly present some of the numerical calculations performed using the 
methods described in Appendix \ref{Appendix}.

Figure \ref{fig:bifcurve1} depicts the bifurcation branch starting at $(c_0,  \varphi_0)$, along with a close-up of the turning point and the end
of the bifurcation curve, which was found using the 
condition that $\varphi=\gamma$. 
The wave speed decreases initially, indicating a subcritical
pitchfork bifurcation that will be rigorously established in Proposition~\ref{prop:local}.
A further observation is that the wave speed $c({\tau})$ along the global bifurcation
curve is contained in a compact subinterval $(0,1)$. 
This will be a key element in characterizing the global structure
of the bifurcation curve, in particular when showing that it does not form a closed loop,
as well as when demonstrating that the limiting wave at the end of the bifurcation
curve is nontrivial, see Lemmas~\ref{lemma:uniqueness} and~\ref{lemma:c_uniformbound}.   
The structure of the bifurcation curve itself, specifically the existence of a turning point, is so far not understood.

\begin{figure}[t]
\begin{center}
\includegraphics[scale=0.5]{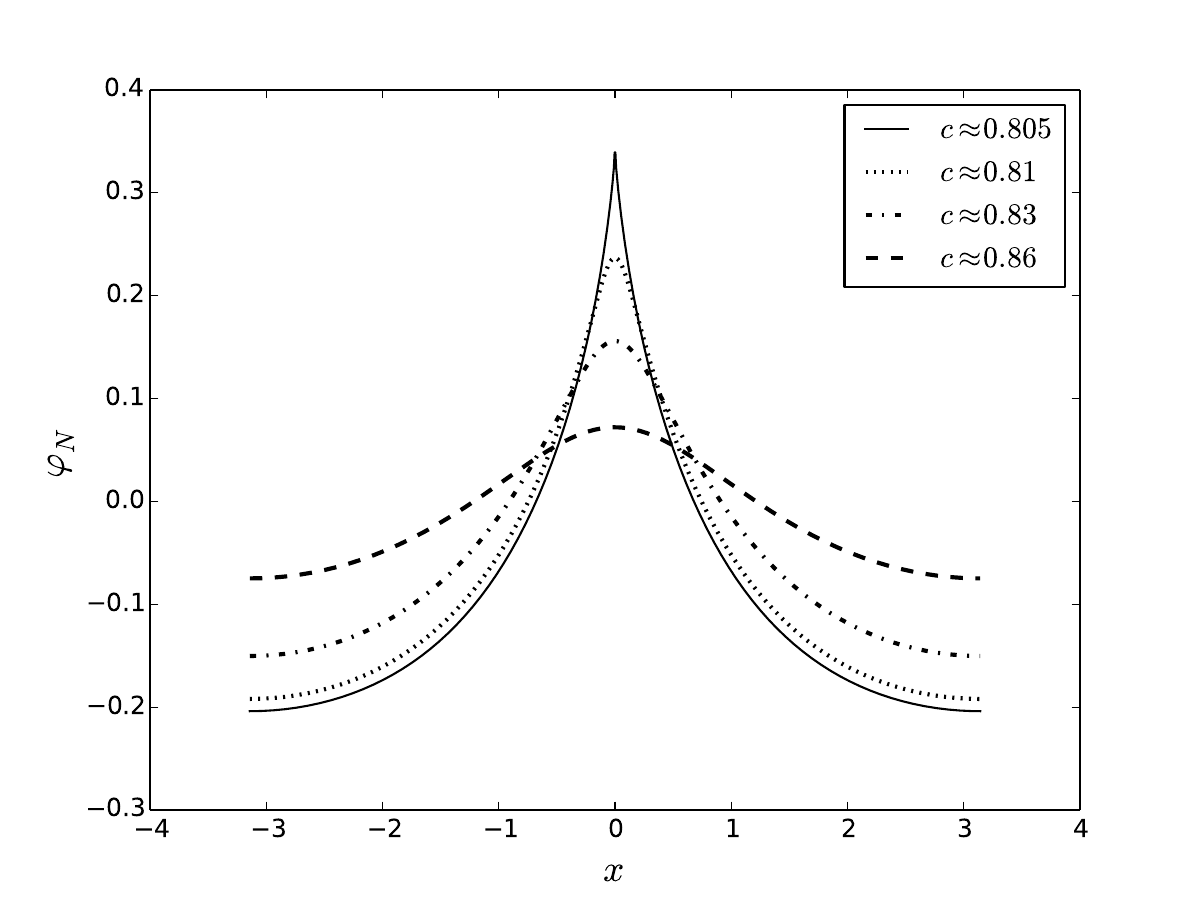}
\end{center}
\caption{\small Four different numerical profiles are shown along the bifurcation branch from Figure \ref{fig:bifcurve1}.
}
\label{fig:bifcurve2}
\end{figure}

Solutions $ \varphi$ of the profile equation \eqref{profile} with wave speed $c$ along the bifurcation curve are depicted
in Figure \ref{fig:bifcurve2}.  The increasing amplitude of the waves corresponds to moving farther up the global
bifurcation branch in Figure \ref{fig:bifcurve1}.  The small-amplitude waves are perturbations of a multiple of $  \cos{(x)}$, a fact that is consistent with the bifurcation formulas derived in Proposition~\ref{prop:local}.  As one continues along the
curve, however, the solutions become increasingly nonlinear and the local theory from Proposition
\ref{prop:local} does not yield any predictions.  From the numerical calculations
we can nevertheless make some observations also for large amplitudes.  First, it appears that solutions along the global bifurcation
curve depicted in Figure~\ref{fig:bifcurve1} are smooth and strictly increasing on a half-period with unique critical points on $[-\pi,\pi)$ given at $x=-\pi$ (minimum) and at $x=0$ (maximum).  That solutions along the bifurcation curve indeed admit such
a nodal pattern is established in Theorem~\ref{thm:nodal} below.

Next, we observe from Figure~\ref{fig:bifcurve2} that solutions become 
progressively steeper at their global maximum as the end 
of the bifurcation branch in Figure~\ref{fig:bifcurve1} is approached.  
In particular, it appears that the derivative of the limiting wave blows up at $x=0$, corresponding
to the limiting wave having a singularity at that point.  While the smoothness away from $x=0$ will be established in Lemma \ref{lemma:uniform_c-bound}, the behavior at the {crest} is more subtle. Whitham reasoned in 
\cite[p.479]{Whitham_book} that if $K_p(x)$ were to blow up 
like $|x|^{-q}$ at $x=0$ for some $q>0$, then a rudimentary scaling
analysis would suggest that the associated solution would behave like \(\gamma-{\varphi}(x) {\eqsim} |x|^{1+\delta}\) with $\delta=-q$. Since $K_p$ has a \emph{logarithmic} blowup at $x=0$, however, such a scaling analysis
{is inadequate} and a more delicate investigation is needed. A detailed investigation of the singularity
at $x=0$ of the limiting wave is the subject of Lemma~\ref{lemma:lower_blowup improved}
and the main regularity result{, Theorem~\ref{thm:reg},} below.

Although not the point of our current investigation, we point out that the dynamic stability of the periodic waves numerically 
constructed above has been recently considered.  In \cite{P17}, the author rigorously derives, using spectral perturbation theory for the linearized spectral problem, an analytical
stability index whose sign determines the modulational (spectral) stability of periodic traveling wave solutions  of \eqref{biw} with asymptotically small amplitude to localized perturbations on the line.  
Outside of this analysis of the asymptotically small waves, there is no rigorous analysis concerning the stability of solutions of \eqref{biw}.  We consider
the stability of waves in these and more general full-dispersion models outside of the small-amplitude regime as an important open problem.
We note, however, the recent work \cite{CJ17} where the global global bifurcation and spectral stability of 
large amplitude waves of \eqref{biw}, and other related bidirecitonal full-dispersion water wave models, have 
been numerically investigated. The interested reader is referred to this paper
for a number of numerical observations concerning the stability of large amplitude waves that is so far unproven.

\begin{remark}\label{signdefinite}
An important observation is that the waves constructed in this paper are not sign definite, which has important consequences relating to the local dynamics about such waves in the evolutionary PDE \eqref{dim_biw}.  Indeed, \eqref{dim_biw} is known to be locally well-posed in standard Sobolev spaces only provided that the Cauchy data has strictly positive surface elevation. The work \cite{CJ17} provides numerical evidence that this surface elevation restriction is sharp. This evolutionary perspective motivates the search also for periodic traveling wave solutions of \eqref{dim_biw} with strictly positive wave height. In \cite{CJ17}, such waves with asymptotically small amplitude were shown to bifurcate from a non-zero equilibrium state of \eqref{dim_biw} through a local bifurcation argument, and numerically continued through the global bifurcation branch of waves with strictly positive waveheight, terminating (numerically) at the line $\max(\varphi)=\gamma$ 
in a highest, cusped and elsewhere smooth traveling wave solution. The extension of our theory to such waves is described in Appendix \ref{appendix_b}.
\end{remark}

\section{A priori properties of solutions \(\varphi \leq \gamma\)}\label{sec:a priori}

We now study periodic solutions of \eqref{profile} in an appropriate subspace of $C^\alpha(\mathbb{S})$
with $\alpha\in(0,1)$.  
By a solution of \eqref{profile} we shall mean a \(2\pi\)-periodic and continuous function that satisfies the equation pointwise.  In our search for a highest wave, we will begin in Section \ref{sec:bifurcation} below by first
constructing small amplitude periodic traveling wave solutions of \eqref{profile} via a local bifurcation
argument.  These small amplitude solutions will then be continued into a global curve of large amplitude
solutions, eventually terminating into a highest wave with a cusp.  As a first step then, we begin
by studying {a priori} properties of solutions with $\varphi<\gamma$ uniformly in $x$, including in particular
the small amplitude solutions constructed via the local theory.  We end with an {a priori} estimate
on even, nondecreasing solutions which achieve the maximum height $\gamma$ at their crests, showing in particular
that such solutions cannot {be} continuously differentiable, {or} even Lipschitz at $x=0$,
and studying the global regularity of such a wave.

\

We start by noting that there are exactly three curves of trivial solutions of \eqref{profile}, namely,
\[
c \mapsto 0 \quad\text{and }\quad c \mapsto \Gamma_{\pm}(c) := \frac{3c \pm \sqrt{8 + c^2}}{2}.
\]

The latter two are reflections of each other around the diagonal \(\varphi = c\), since the map 
\[
(\varphi,c) \mapsto -(\varphi,c)
\]
describes a bijection between solutions with positive and negative wave speed. \emph{For that reason, it is enough to restrict our attention to \(c \geq 0\).} In particular, we shall primarily be concerned with pairs \(  (\varphi,c)\) such that \(   (\sup \varphi,c)\) lies in the area enclosed by \(c = 0\), \(c=1\), \(\sup \varphi = 0\) and \(\sup \varphi = \gamma\). The curve \(\Gamma_+\) is outside of this domain and will therefore not be relevant in our analysis. The curve \(\Gamma_-\), however, crosses \(c=0\) at \(\varphi = -\sqrt{2}\) and the line of zero solutions at \(c = 1\) (whereafter it reaches  \(\varphi = \gamma\) at \(c \approx \frac{3}{2}\)). We will have to deal with this fact in our limiting argument.

\begin{figure}[t]
\begin{center}
\includegraphics[width=0.8\linewidth]{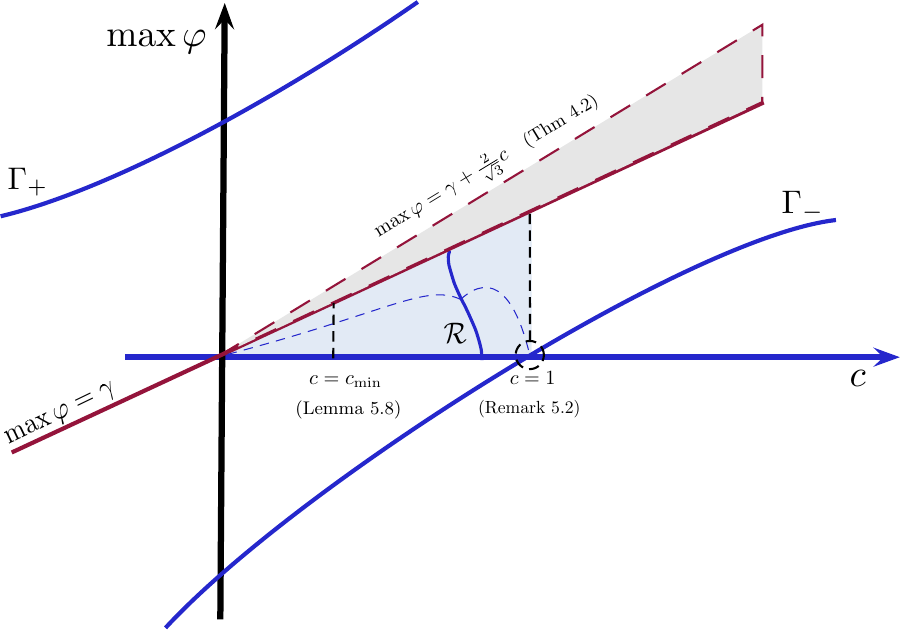}
\end{center}
\caption{
 \small The global bifurcation diagram obtained in Theorem~\ref{thm:global}. The solid blue lines indicate solutions, and the red upper bounds on \(\varphi\). The solution curves \(\Gamma_{\pm}\) and the horizontal axis correspond to constant solutions, whereas \(\mathcal R\) is the global bifurcation curve leading to the highest wave. All solutions studied in this paper are confined to the region enclosed by the lines \(\max \varphi = 0\), \(\max \varphi = \gamma\) and \(c = 1\). The transcritical bifurcation at \(c=1\) and the a priori bound on the wavespeed (established in Remark~\ref{rem:local} and Lemma~\ref{lemma:c_uniformbound}, respectively) exclude solutions branching off to zero as in the picture.} 
\end{figure}

\begin{lemma}\label{lemma:uniform_c-bound}
For all solutions \(\varphi\) of \eqref{profile}, one has the uniform estimate
\[
\|\varphi\|_{\infty} \lesssim 1 + c.
\]
Furthermore, solutions of \eqref{profile} are smooth on any open set where \(\varphi(x)< \gamma\).
\end{lemma}

\begin{proof}
Lemma~\ref{p:kernel} implies that $\mathcal{K}\in\mathcal{L}(L^\infty(\mathbb{S}))$, with unit operator norm. 
Rewriting \eqref{profile} as
\[
{\textstyle \frac{1}{2}} \varphi^3=\mathcal{K}\varphi-c^2\varphi+{\textstyle \frac{3c}{2}} \varphi^2,
\]
we see that either $\varphi\equiv 0$ or else
\[
{\textstyle \frac{1}{2}}\|\varphi\|_{L^\infty}^2-{\textstyle \frac{3c}{2}}\|\varphi\|_{L^\infty}-(1+c^2)\leq 0.
\]
In either case, it follows that
\[
\|\varphi\|_{L^\infty}\leq\frac{3c+\sqrt{8+c^2}}{2}\lesssim 1+c,
\]
as claimed.

To prove smoothness, assume first that \(\max_x \varphi(x) < \gamma\). Recall that \(\gamma\) is the smallest root of $N'(z)=0$, whence the inverse function theorem guarantees the existence of a smooth function \(N^{-1}\) such that
\[
N^{-1} N(\varphi) = \varphi \qquad \text{for }\quad   m_1 \leq \varphi \leq m_2 < \gamma.
\] 
Since $\mathcal{K}: C^s_*(\mathbb{S}) \to C^{s+1}_*(\mathbb{S})$ for all $s \in \RM$, the Nemytskii operator
\begin{equation}\label{eq:nemytskii}
\varphi\mapsto N^{-1} (\mathcal{K} \varphi)
\end{equation}
maps \(C^s_*(\mathbb{S})\) into \(C^{s+1}_*(\mathbb{S})\) for all $s \in \RM$. If \(\varphi \in C(\mathbb{S}) \hookrightarrow C_*^0(\mathbb{S})\) is now a given solution with \(\max_x \varphi(x) < \gamma\), it follows by induction that \(\varphi\) is in fact smooth.

Now, if \(\varphi(x) < \gamma\) in an open ball \(B_{\varepsilon_0}(x_0)\), we write \(\varphi = \varphi \psi + \varphi (1-\psi)\) for a smooth function \(\psi\) with  \(\psi(x) = 1\) for \(x \in B_{\varepsilon'}(x_0)\) for a slightly smaller \(\varepsilon' < \varepsilon\), and \(\supp(\psi) \Subset B_{\varepsilon}(x_0)\). The term \(\varphi \psi\) has the same regularity as \(\varphi|_{B_{\varepsilon}}\), globally on \(\mathbb{S}\). As what concerns the second term, we note that
\[
K_p(x-y)\varphi(y) (1-\psi(y)) = 0 \quad\text{for }\quad y \in  B_{\varepsilon'}(x_0) + 2\pi \ZM.
\]
Since \(K_p(x-y)\) is smooth for \(x - y \not\in 2\pi\ZM\), the convolution \(\mathcal{K} \varphi (1- \psi)(x)\) is smooth for \(x \in B_{\varepsilon'}(x_0)\). Taken together, if \(\varphi \in C^s_*(B_{\varepsilon})\), we have \(\mathcal{K} \varphi \psi \in C^{s+1}_*(B_{\varepsilon'})\). Since \(\varepsilon' < \varepsilon\) is arbitrary, we conclude that \(\mathcal{K} \varphi \psi \in C^{s+1}_*(B_{\varepsilon})\). Thus, if \(\sup_x \varphi(x) < \gamma\) in \(B_\varepsilon(x_0)\), we may apply the Nemytskii operator \eqref{eq:nemytskii} repeatedly to obtain smoothness of \(\varphi\) in the same set.
\end{proof}

A key ingredient in our forthcoming global bifurcation theory will be the preservation of a particular nodal
pattern for solutions of \eqref{profile} that satisfy $\varphi<\gamma$ uniformly in $x$.  This is the content
of the following technical result.

\begin{theorem}\label{thm:nodal}
Any non-constant and even solution \(\varphi \in C^1(\mathbb{S})\) of \eqref{profile} which is non-decreasing on \((-\pi,0)\) and satisfies \(\max \varphi  \leq \gamma + \frac{2}{\sqrt{3}}c\) fulfills 
\[\varphi' > 0, \: \varphi < \gamma \qquad\text{on }\quad (-\pi,0).
\]
If \(\varphi \in C^2(\mathbb{S})\), then \(\varphi < \gamma \) everywhere, with
 \[ 
 \varphi^{\prime\prime}(0) < 0, \quad \varphi^{\prime\prime}(\pm \pi) > 0,  
 \]
and \(\varphi^{\prime\prime}(\pi) -  \varphi^{\prime\prime}(0) \gtrsim \varphi(0) - \varphi(\pi)\).
\end{theorem}

\begin{proof}
Since $\varphi$ is even and non-constant, we have that $\varphi'$ is odd, non-trivial, and non-positive on $(0,\pi)$.
We claim that $\mathcal{K}\varphi'(x)<0$ for all $x\in(0,\pi)$.  To see this, notice that {the evenness of the} periodic kernel $K_p$ gives
\begin{align*}
\mathcal{K}\varphi'(x)&=\int_{-\pi}^\pi K_p(x-y)\varphi'(y)dy\\
&=\int_0^\pi\left[K_p(x-y)-K_p(x+y)\right]\varphi'(y)dy.
\end{align*}
{Furthermore,} since
\[
K_p(x-y)-K_p(x+y)>0 \quad \text{for all }  \quad  x,y\in(0,\pi),
\]
so as long as $\varphi$ is non-constant, we have $\mathcal{K}\varphi'(x)<0$
for all $x\in(0,\pi)$, as claimed.  Now note that, by \eqref{eq:N_doubleroot},
\begin{align}\label{eq:profile_derivative}
{\textstyle \frac{3}{2}} (\gamma - \varphi ) ( \gamma  -\varphi + {\textstyle \frac{2}{\sqrt{3}}} c ) \varphi^\prime& = N'(\varphi)\varphi' = \mathcal{K}\varphi' < 0,
\end{align}
on \((0,\pi)\). Since by assumption \( \varphi' \leq 0\) on this interval, we first get the strict inequality \( \varphi^\prime < 0\) on \((0,\pi)\). Then \(\left[\gamma - \varphi \right] [ (\gamma  -\varphi)  + {\textstyle \frac{2}{\sqrt{3}}} c ] > 0\) on the same interval, which holds exactly when \(\varphi < \gamma\)  or \(\varphi > \gamma + \frac{2}{\sqrt{3}}c\).
The second alternative is excluded by assumption, whence we conclude that \(\varphi < \gamma\).

Now, if \(\varphi \in C^2(\mathbb{S})\), one obtains from \eqref{eq:profile_derivative} that 
\begin{equation}\label{eq:profile_twoderivatives}
\begin{aligned}
N^\prime(\varphi(x)) \varphi^{\prime\prime}(x) &=  \int_0^\pi\left[K_p^\prime(x-y)-K_p^\prime(x+y)\right]\varphi'(y) \, dy, \quad x \in \pi \ZM,
\end{aligned}
\end{equation}
where the integral is well defined for \(x \in \pi \ZM\) in view of \eqref{eq:Kp_twoparts} and the fact that {\(\varphi^\prime(x+z) = \bigO(z)\) for \(x \in \pi\ZM\), \(|z| \ll 1\)}. When $x=0$, we get
\begin{align*}
N^\prime(\varphi(0)) \varphi^{\prime\prime}(0) &=  -2\int_0^\pi K_p^\prime(y) \varphi'(y) \, dy < 0, 
\end{align*}
and for \(x = \pi\),
\begin{align*}
N^\prime(\varphi(\pi)) \varphi^{\prime\prime}(\pi) &=  -2\int_0^\pi K_p^\prime(\pi  + y) \varphi'(y) \, dy > 0, 
\end{align*}
yielding \(\varphi^{\prime\prime}(0) < 0\) and \(\varphi^{\prime\prime}(\pm\pi) > 0\) (the strict inequality also yields that \(\varphi < \gamma\) everywhere). Furthermore,
\begin{align*}
&N^\prime(\varphi(\pi)) (\varphi^{\prime\prime}(\pi)- \varphi^{\prime\prime}(0)) + \varphi^{\prime\prime}(0) (N^\prime(\varphi(\pi)) - N^\prime(\varphi(0))\\ 
&=  2\int_0^\pi  (K_p^\prime(y) -K_p^\prime(\pi  + y)) \varphi'(y) \, dy\\
& \gtrsim K_{p}^\prime(-\textstyle{\frac{\pi}{2}}) (\varphi(0) - \varphi(\pi)), 
\end{align*}
by the concavity of \(K_p\). Since \(N^\prime(\varphi(\pi))\) is positive and bounded, \(\varphi^{\prime\prime}(0) < 0\) and \(N^\prime(\varphi(\pi)) - N^\prime(\varphi(0) > 0\), the estimate \(\varphi^{\prime\prime}(\pi) -  \varphi^{\prime\prime}(0) \gtrsim \varphi(0) - \varphi(\pi)\) follows.
\end{proof}

By the above result, all even,  $2\pi$-periodic smooth
solutions that are nondecreasing on $(-\pi,0)$
with $\varphi\leq \gamma$ on $\mathbb{R}$ 
are smooth on $(0,\pi)$ and are strictly decreasing on {the same interval} with $\varphi(0)<\gamma$.  
In the next result, we allow for the possibility
that $\varphi(0)=\gamma$ and study the behavior of such a solution near $x=0$.

\begin{lemma}\label{lemma:lower_blowup improved}
Let $\varphi$ be an even, non-constant, $2\pi$-periodic solution of \eqref{profile} 
such that $\varphi$ is nondecreasing
on $(-\pi,0)$ with $\varphi\leq\gamma$ on $(-\pi,\pi)$.  Then $\varphi$  is smooth 
and strictly increasing on $(-\pi,0)$, and as $x\to 0$ we have
\begin{equation}\label{eq:lower_blowup}
\gamma-\varphi(x)\gtrsim\frac{|x\log|x||}{1+c}.
\end{equation}
\end{lemma}

\begin{proof}
First, note by Lemma \ref{lemma:uniform_c-bound} and Theorem \ref{thm:nodal} that if 
$\varphi(0)<\gamma$ then $\varphi\in C^\infty(\RM)$ and is strictly increasing on $(-\pi,0)$.
In the case when $\varphi(0)=\gamma$, however, $\varphi$ may not be $C^1$ and hence we cannot
establish smoothness nor strict monotonicity as above.  Nevertheless, we now prove that,
just as in Theorem~\ref{thm:nodal}, one has $\varphi'(x)>0$ for all $x\in(-\pi,0)$ 
even when \(\varphi\) is merely assumed to be continuous.
This is a technical variation of the argument used in the proof of Theorem~\ref{thm:nodal}, which starts with the observation that
\begin{equation}\label{eq:double sym}
\begin{aligned}
&{\mathcal K}\varphi(x+h) - {\mathcal K}\varphi(x-h)\\ 
&\quad = \int_{-\pi}^0 \left[ K_p(y-x) - K_p(y+x) \right]  \left[ \varphi(y+h) - \varphi(y-h) \right] \, dy,
\end{aligned}
\end{equation}
in view of evenness and periodicity of both \(\varphi\) and \(K_p\). For \(x \in (-\pi,0)\) and \(h \in (0,\pi)\), both factors in the integrand are non-negative, and since \(\varphi\) is assumed to be non-constant, we conclude that \({\mathcal K}\varphi(x+h) > {\mathcal K}\varphi(x-h)\) whenever \(x,h\) are chosen as above. From \eqref{eq:profile'} we have
\begin{equation}\label{eq:difference expansion}
\begin{aligned}
& {\mathcal K}\varphi(x) - {\mathcal K}\varphi(y)\\ 
&= ( \varphi(x) - \varphi(y) ) \left( c^2 - {\textstyle \frac{3c}{2}} (\varphi(x) + \varphi(y)) + {\textstyle \frac{1}{2}} \left( (\varphi(x))^2 + \varphi(x) \varphi(y) + (\varphi(y))^2 \right) \right).
\end{aligned}
\end{equation}
By letting \(\tilde \varphi = c \varphi\), one sees that, up to a factor of \(c^2/2\), the long expression on the right-hand side is non-negative because
\[
2 - 3 (\tilde \varphi(x) + \tilde \varphi(y)) + (\tilde\varphi(x))^2 + \tilde \varphi(x) \tilde \varphi(y) + (\tilde \varphi(y))^2 \geq 0
\] 
for \(\tilde \varphi \leq 1 - 1/\sqrt{3}\), with equality only when \( \tilde \varphi(x) = \tilde \varphi(y) = 1- 1/\sqrt{3}\) (recall here that \(\varphi \leq \gamma\)). Since we already proved that \(\mathcal{K}\varphi\) is strictly increasing on \(-\pi,0)\), the assumption that \(\varphi\) is nondecreasing together with  \eqref{eq:difference expansion} show that \(\varphi\) is indeed strictly increasing on \((-\pi,0\)) (hence, \(\varphi(x) = \gamma\) is excluded except at \(x \in 2\pi \ZM\)).  Consequently, $\varphi$ is smooth on $(-\pi,0)$ by Lemma~\ref{lemma:uniform_c-bound}, and to conclude that \(\varphi\) also has a strictly positive derivative in the left half-period, one may apply Fatou's lemma to \eqref{eq:double sym}. Via \eqref{eq:difference expansion} this shows that \(\varphi'(x) > 0\) for \(x \in (-\pi,0)\). 
 
To establish the lower bound \eqref{eq:lower_blowup}, observe that by Lemma \ref{lemma:uniform_c-bound} we have
\begin{equation}\label{eq:Nest2}
(1+c)(\gamma-\varphi(x))\gtrsim\frac{3}{2}(\gamma-\varphi(x))(\gamma-\varphi(x)+\frac{2}{\sqrt{3}}c)=N'(\varphi(x))
\end{equation}
so that it is sufficient to study the behavior of $N'(\varphi(x))$ for $|x|\ll 1$.
From the above, for each $x\in[-\pi,0)$ there exists a $\xi\in(x,x/2)$ such that
\[
\varphi'(\xi)=\min_{y\in[x,x/2]}\varphi'(y).
\]
Since $N'(\varphi)$ is strictly decreasing in $\varphi$ for all $\varphi<\gamma$, it follows
from the monotonicity of $\varphi$ that $N'(\varphi(x))\geq N'(\varphi(\xi))$ for all $x\in[-\pi,0)$.
From \eqref{profile} and the fact that $\varphi'(x)>0$ for $x\in(-\pi,0)$ {we then see that}
\begin{align*}
N'(\varphi(x))\varphi'(\xi)&\geq N'(\varphi(\xi))\varphi'(\xi)\\
&=\mathcal{K}\varphi'(\xi)\\
&=\int_{-\pi}^0\left[K_p(\xi-y)-K_p(\xi+y)\right]\varphi'(y) \, dy\\
&\geq\varphi'(\xi)\int_{x}^{\frac{x}{2}}\left[K_p(\xi-y)-K_p(\xi+y)\right]\, dy
\end{align*}
where the last inequality follows {by} $K_p(\xi-y)-K_p(\xi+y)>0$ for $\xi,y\in(-\pi,0)$,
and {by}  the definition of $\xi$.
Lemma \ref{lemma:Kp}(iii) now immediately provides the estimate
\begin{align*}
N'(\varphi(x))&\geq  \int_{x}^{\frac{x}{2}} \left[K_p(\xi-y)-K_p(\xi+y)\right]\,dy\\
&\geq \max \left(\int_{x}^{\xi}, \int_{\xi}^{\frac{x}{2}} \right)   
\left[K_p(\xi-y)-K_p(\xi+y)\right]\,dy\\
&\geq 2\int_I \left[\log|  \pi(\xi+y)/4|-\log|\pi(\xi-y)/4|\right]\,dy - \bigO(x^3).
\end{align*}
{Here we take} $I=[\xi,x/2]$ when $\xi\leq\frac{3}{4}x$ and $I=[x,\xi]$ when $\xi>\frac{3}{4}x$.
It follows that
\begin{equation}\label{eq:Nest}
N'(\varphi(x))\gtrsim  |(\xi - z) \log |\pi (\xi - z)| | - |(\xi + z) \log |\pi(\xi+z)| | - \bigO(x)
\end{equation}
for all $x\in(-\pi,0)$, {where} $z=x$ if $\xi>\frac{3}{4}x$ and $z=\frac{x}{2}$ if $\xi\leq\frac{3}{4}x$.
Considering \(|x| \ll 1\) small enough for \(|x \log | \pi x||\) to be monotone in \(x\),
we find that for such $x$ and $z$ we have
\begin{equation}\label{eq:improved xlog estimate 2}
|(\xi - z) \log |\pi (\xi - z)| | - |(\xi + z) \log |\pi(\xi+z)| | \geq |{\textstyle \frac{x}{4}} \log |\pi {\textstyle \frac{x}{4}}| - |x \log |\pi x| |,
\end{equation}
and we may further estimate
\begin{equation}\label{eq:improved xlog estimate 3}
\begin{aligned}
|{\textstyle \frac{x}{4}} \log | \pi {\textstyle \frac{x}{4}}| - |x \log |\pi x| | &= \left| \frac{x}{4}  \log\left| \frac{\pi x/4}{(\pi x)^4} \right| \right|\\
&= \left| \frac{3x}{4}  \log\left| 4^{1/3} \pi x \right| \right| \geq \frac{3}{4} |x \log |x|| - \bigO(x).
\end{aligned}
\end{equation}
Combining 
{\eqref{eq:Nest}, \eqref{eq:improved xlog estimate 2} and \eqref{eq:improved xlog estimate 3}}, the result follows immediately by \eqref{eq:Nest2}.
\end{proof}

The estimate \eqref{eq:lower_blowup} obviously holds for any solution \(\varphi \in C(\mathbb{S})\) than can be approximated in \(C(\mathbb{S})\) by a sequence of solutions satisfying the assumptions of the lemma. In particular, if \(\varphi(0) =\gamma\) for such a solution, Lemma~\ref{lemma:lower_blowup improved} implies that the solution cannot be continuously differentiable, or even Lipschitz continuous, at $x=0$.  The next result
explores the global regularity of such a wave, as well as the singularity at $x=0$ {(we so far only have a lower bound on \(\gamma - \varphi(x)\))}.

\begin{theorem}\label{thm:reg}
Let $\varphi$ be an even, $2\pi$-periodic solution of \eqref{profile} that is nondecreasing on $(-\pi,0)$
with $\varphi(0)=\gamma$.  Then:
\begin{itemize}
\item[(i)] $\varphi$ is smooth and strictly increasing on $(-\pi,0)$.
\item[(ii)] $\varphi\in C^\alpha(\mathbb{S})$ for all $\alpha\in(0,1)$, and the \(C^\alpha\)-estimates are uniform in \(\alpha\) over any compact subset of \((0,1)\), and uniform in \(\varphi\) for wavespeeds \(c\) contained in any compact subset of \((0,\infty)\).
\item[(iii)]  The estimate
\begin{equation}\label{eq:reg}
\gamma-\varphi(x)\eqsim \left|x\log|x|\right|,
\end{equation}
holds for all $|x|\ll 1$.
\end{itemize}
\end{theorem}

\begin{proof}
Part (i) and the lower bound in \eqref{eq:reg} have already been established in 
 Lemma~\ref{lemma:lower_blowup improved}.  It thus remains to {prove} the global regularity
result in (ii) and the upper bound in (iii).

To establish (ii), let \(0 \leq x < y \leq \pi\) and note that {by} Taylor's theorem,
\[
N(\varphi(x)) - N(\varphi(y))
= (\varphi(x) - \varphi(y)) N^\prime(\varphi(x)) - {\textstyle \frac{1}{2}} (\varphi(x) - 
\varphi(y))^2 N^{\prime\prime}(\varphi(\xi_1)),
\]
for some $\xi_1\in(x,y)$.
Further, using $N'(\gamma)=0$, the {mean value theorem} implies
\[
N'(\varphi(x)) =N''(\varphi(\xi_2)(\varphi(x)-\gamma)
\]
for some $\xi_2\in(0,x)$, so that
\begin{align*}
N(\varphi(x))-N(\varphi(y))&= -(\varphi(x) - \varphi(y)) (\gamma - \varphi(x)) N^{\prime\prime}(\varphi(\xi_2))\\ 
&\qquad- {\textstyle \frac{1}{2}} (\varphi(x) - \varphi(y))^2 N^{\prime\prime}(\varphi(\xi_1)).
\end{align*}
Now note that
\[
N''(\varphi(\xi)) \leq - \sqrt{3} c \quad\text{for all }\quad \xi \in [0,\pi],
\]
in view of that \(N''(\gamma) = - \sqrt{3} c\) and that, for such $\xi$, we have 
\(\frac{d}{d\xi}N''(\varphi(\xi)) < 0\) and \(\varphi(\xi) \leq \gamma\). In particular,
\[
N(\varphi(x))-N(\varphi(y))\gtrsim (\varphi(x)-\varphi(y))(\gamma-\varphi(x))+{\textstyle \frac{1}{2}}(\varphi(x)-\varphi(y))^2,
\]
which holds uniformly  for all solution pairs \((\varphi, c)\) with \(c \gtrsim 1\).
Since $\varphi$ is monotone decreasing on $(0,\pi)$, it follows from \eqref{profile} that the above estimate
yields
\begin{equation}\label{eq:keyestimate1}
{\mathcal{K}\varphi(x)-\mathcal{K}\varphi(y)} \gtrsim(\varphi(x)-\varphi(y))(\gamma-\varphi(x))
\end{equation}
and
\begin{equation}\label{eq:keyestimate2}
\mathcal{K}\varphi(x)-\mathcal{K}\varphi(y)\gtrsim (\varphi(x)-\varphi(y))^2,
\end{equation}
uniformly for \(c \gtrsim 1\). Now, recall that if $\varphi\in C^s_*(\mathbb{S})$ for some $s\geq 0$, 
then $\mathcal{K}(\varphi)\in C^{s+1}_*(\mathbb{S})$.  From \eqref{eq:keyestimate2} and the continuity
of the embedding $ C^1_*(\mathbb{S}) \hookrightarrow C^\alpha(\mathbb{S})$ for all $\alpha\in[0,1)$, it {is immediate} that any solution $\varphi\in C^0(\mathbb{S})$ belongs to $C^{1/2-}(\mathbb{S})$.  

To show that $ \varphi$ has better regularity than $C^{1/2-}(\mathbb{S})$ we observe that for any
\(f \in C^{1+\alpha}(\mathbb{S})\) with $\alpha\in(0,1)$ and \(f^\prime(0) = 0\), one has the estimate
\begin{align*}
|f(x) - f(y)| &=  |x-y| | f^\prime(\xi) - f^\prime(0)| \lesssim |x-y| |\xi|^\alpha,
\end{align*}
valid for {some} \(|\xi| \in (x,y)\).  Applying this estimate to the function $\mathcal{K}  \varphi$, it follows that 
if ${\varphi} \in C^\alpha(\mathbb{S})$ for some $\alpha\in(0,1)$, then for $0\leq x< y\leq\pi$ we have
\begin{equation}\label{eq:keyestimate3}
\mathcal{K} \varphi(x) - \mathcal{K} \varphi(y) \lesssim |x-y| y^\alpha.
\end{equation}
Whenever {$\varphi \in C^\alpha(\mathbb{S})$, \(\alpha \in (0,1)\), the} estimates \eqref{eq:keyestimate2}, \eqref{eq:keyestimate3} 
and the triangle inequality together yield
{
\begin{equation}\label{eq:peak-estimate}
\varphi(x)-\varphi(y)\lesssim |x-y|^{\frac{1+\alpha}{2}}
\end{equation}
valid uniformly for all $0\leq x<y\leq\pi$ with $x<|x-y|$, and all solutions \(\varphi\) with \(c \gtrsim 1\)}.
In particular, taking $x=0$ above implies
\begin{equation}\label{eq:top-estimate}
\gamma - \varphi(y)  \lesssim |y|^{\frac{1+\alpha}{2}},
\end{equation}
for all \(y \in \mathbb{S}\), when \(\varphi \in C^{\alpha}(\mathbb{S})\) {for \(\alpha \in (0,1)\).}
When, on the other hand, $|x-y|\leq x$ we have from \eqref{eq:keyestimate1}, \eqref{eq:keyestimate3} and
the triangle inequality that
\[
 \left(\varphi(x) - \varphi(y) \right) \left(\gamma - \varphi(x) \right) \lesssim |x-y| x^\alpha.
\]
Since \( \gamma - \varphi(x) \gtrsim x/(1+c)\) by Lemma \ref{lemma:lower_blowup improved}, it follows that
\begin{equation}\label{eq:xh-estimate}
\varphi(x) - \varphi(y)  \lesssim \frac{|x-y|}{x^{1-\alpha}},
\end{equation}
whenever \(\varphi \in C^\alpha(\mathbb{S})\) for some $\alpha\in(0,1)$, {and \(c \eqsim 1\).}

We now interpolate between \eqref{eq:top-estimate} and~\eqref{eq:xh-estimate}, still for $|x-y|\leq x$. 
Namely, {using that \(y < 2x\),} for a given $\beta\in(0,1)$ we estimate
\begin{align*}
\frac{\varphi(x) - \varphi(y)}{|x-y|^\beta} &\leq \frac{(\varphi(x) - \varphi(y))^{\beta}}{|x-y|^\beta} 
(\gamma - \varphi({y}))^{1-\beta}\\ 
&\lesssim x^{(\alpha - 1)\beta + \frac{(1+\alpha)(1-\beta)}{2}},
\end{align*}
which is bounded for all $0\leq x<y\leq\pi$ provided that \(\beta \leq \frac{1+\alpha}{3-\alpha}\).  
In particular, taking 
$\beta=\frac{1+\alpha}{3-\alpha}$ above we have the estimate
\[
{\varphi}(x)-{\varphi}(y)\lesssim|x-y|^{(1+\alpha)/(3-\alpha)}
\]
when $|x-y|\leq x$, valid {uniformly for all solutions $\varphi \in C^\alpha(\mathbb{S})$ for which \(c \eqsim 1\). Here, $\alpha\in(0,1)$ is still considered fixed.}  Combining with
\eqref{eq:peak-estimate} and noting that $2<3-\alpha$, we have established the estimate
\[
{\varphi}(x)-{\varphi}(y)\lesssim|x-y|^{(1+\alpha)/2}\quad\text{for all}\quad 0\leq x<y\leq\pi
\]
whenever ${\varphi}\in C^{\alpha}(\mathbb{S})$ with $\alpha\in(0,1)$.  It follows
that if ${\varphi}\in C^\alpha(\mathbb{S})$ is a solution of \eqref{profile} for some $\alpha\in(0,1)$,
then ${\varphi}\in C^{(1+\alpha)/2}(\mathbb{S})$.  Fixing $\alpha_0\in(0,1/2)$, {we may} define the recurrence relation
\[
a_0=\alpha_0,\quad a_{n+1}=\frac{1+a_n}{2},~~n\geq 0,
\]
{yielding that $\varphi\in C^{a_n}(\mathbb{S})$} for all $n\in\mathbb{N}$.  Since the sequence $\{a_n\}_{n=1}^\infty$
is clearly strictly increasing with $a_n\nearrow 1$, {$\varphi$ belongs to $C^{1-}(\mathbb{S})$,}
as claimed. {The \(C^\alpha\)-estimates are furthermore uniform for all \(\alpha\) in any compact subinterval of \((0,1)\), and for all solution pairs \((\varphi,c)\) with \(c\) in a compact subinterval of \((0,\infty)\).}

It remains to establish the upper bound in \eqref{eq:reg}. {This is the most technical part of the paper.}
Observe that since $\varphi\in C^\alpha(\mathbb{S})$ for all $\alpha\in(0,1)$, one has
\begin{equation}\label{eq:calpha}
c_\alpha := \sup_{y \in \mathbb{S}} \frac{\gamma  - \varphi(y)}{|y|^\alpha (1 + |\log|y||)} < \infty
\quad \text{ for all } \quad \alpha\in(0,1),
\end{equation}
and our goal is to show that one may let $ \alpha \nearrow 1$ to obtain the desired bound 
$\gamma-\varphi(x)\lesssim \left|x\log|x|\right|$ for all $x$ sufficiently small.
To this end, let $0<\delta\ll 1$ and 
note for all $x\in(0,\delta)$ we have from \eqref{eq:keyestimate2} {that}
\begin{align*}
\left(\gamma-\varphi(x)\right)^2&\lesssim
\mathcal{K}\varphi(0)-\mathcal{K}\varphi(x)\\
&=\int_{-\pi}^\pi\left(K_p(y)-K_p(x-y)\right)\varphi(y)dy\\
&=\int_{-\pi}^\pi\left(K_p(x-y)-K_p(y)\right)\left(\gamma-\varphi(y)\right)dy.
\end{align*}
{Here} we have used the fact that $K_p$ is $2\pi$-periodic.  Taking $y\mapsto -y$ above and averaging
gives the representation
\begin{align*}
&\mathcal{K}\varphi(0)-\mathcal{K}\varphi(x)=
\frac{1}{2}\int_{-\pi}^\pi\left(K_p(x+y)+K_p(x-y)-2K_p(y)\right)\left(\gamma-\varphi(y)\right)dy\\
&\qquad= -2\int_{0}^{\pi} (\log\left|{\textstyle\frac{x+y}{4}}\right|
		+ \log \left|{\textstyle\frac{x-y}{4}}\right|-2 \log\left|{\textstyle\frac{y}{4}}\right|) 
			(\gamma - \varphi(y))\, dy\\
&\qquad\quad 
+\int_{0}^{\pi} (K_{p,\text{reg}}(x+y)+K_{p,\text{reg}}(x-y)-2K_{p,\text{reg}}(y)) (\gamma - \varphi(y))\, dy,
\end{align*}
where the final equality follows from Lemma \ref{lemma:uniform_c-bound}.  Since $K_{p,{\rm reg}}''$ is uniformly
bounded on $\mathbb{R}$, {the estimate}
\begin{equation}\label{eq:regular_integral}
\int_{0}^{\pi} (K_{p,\text{reg}}(x+y)+K_{p,\text{reg}}(x-y)-2K_{p,\text{reg}}(y)) (\gamma - \varphi(y))\, dy\lesssim x^2
\end{equation}
{holds.} To estimate {also} the principle part, observe {that}
\begin{equation}\label{eq:log-integral}
\begin{aligned}
&\left| \int_{0}^{\pi} (\log\left|{\textstyle\frac{x+y}{4}}\right|
		+ \log \left|{\textstyle\frac{x-y}{4}}\right|-2 \log\left|{\textstyle\frac{y}{4}}\right|)  (\gamma - \varphi(y))\, dy \right|\\
&\leq c_\alpha \int_{0}^{\pi} \left| \log\left|{\textstyle\frac{x+y}{4}}\right|
		+ \log \left|{\textstyle\frac{x-y}{4}}\right|-2 \log\left|{\textstyle\frac{y}{4}}\right|  \right| |y|^\alpha (1 + |\log|y||) \, dy.
\end{aligned}
\end{equation}
Making the change of variables \(y = xs\), we note that
\begin{equation}\label{eq:log-asymptotic}
\begin{aligned}
 &\log|{\textstyle\frac{x(1+s)}{4}}|+ \log |{\textstyle\frac{x(1-s)}{4}}|-2 \log|{\textstyle\frac{xs}{4}}|\\ 
 &=   \log|1+s|+ \log |1-s|-2 \log s \eqsim \frac{1}{s^2} \quad \text {as } s \to \infty,
\end{aligned}
\end{equation}
{whence the magnitude of this expression is independent of \(x\), and is integrable in \(s\) on \((0,\infty)\). 
The integral on the right-hand side in \eqref{eq:log-integral} can thus be further estimated as}
\begin{equation}\label{eq:log-estimate}
\begin{aligned}
&c_\alpha x^{1+\alpha} \int_{0}^{\frac{\pi}{x}} \left| \log|1+s|+ \log |1-s|-2 \log s  \right| s^\alpha (1 + |\log x + \log s|) \, ds \\
&\lesssim c_\alpha x^{1+\alpha} \left( |\log x| + |\log x| \int_{1}^{\frac{\pi}{x}} \frac{s^\alpha}{s^2} \, ds +  \int_{1}^{\frac{\pi}{x}} \frac{s^\alpha \log s}{s^2}\,ds \right)\\
&\lesssim c_\alpha x^{1+\alpha} \left(|\log x| + |\log x| \int_{1}^{\frac{\pi}{x}} \frac{ds}{s}  +  \int_{1}^{\frac{\pi}{x}} \frac{\log s}{s}\,ds \right)\\
&\lesssim c_\alpha x^{2\alpha} |\log x|^2,
\end{aligned}
\end{equation}
where we have used that \(\int_1^{\pi/x} \log(s)/s\,ds = \frac{1}{2}(\log(\pi/x))^2 \lesssim (\log(x))^2\) and that all terms {that are bounded in \(s\)} are \(\bigO(\log(x))\) for \(x \in (0,\delta]\). Specifically, the integrals are 
\(\bigO(|\log(x)|)\) on any compact interval, so it does not matter which starting point we choose for the integrals with end-point \(\frac{\pi}{x}\).  

Finally, combining \eqref{eq:regular_integral}, \eqref{eq:log-estimate} and \eqref{eq:keyestimate2} it follows
that
\begin{align*}
(\gamma - \varphi(x))^2 &\lesssim(\mathcal{K}\varphi)(0)-(\mathcal{K}\varphi)(x) \lesssim c_\alpha  (x^\alpha (1 + |\log x|))^2,
\end{align*}
where the estimates are uniform for all \((\alpha,x) \in (0,1) \times (0,\delta]\), and where \(c_\alpha\) is 
defined in \eqref{eq:calpha}.  Rearranging, the above yields the estimate
\[
\left(\frac{\gamma-\varphi(x)}{x^\alpha(1+|\log x|)}\right)^{2} \lesssim c_\alpha
\]
valid for all $x\in(0,\delta]$ and $\alpha\in(0,1)$.  For $x\in(\delta,\pi]$ we note that the left-hand side
in the above inequality is uniformly bounded for all $\alpha\in[0,1)$, hence we find
\[
\left(\frac{\gamma-\varphi(x)}{x^\alpha(1+|\log x|)}\right)^{2} \lesssim \max\left(c_\alpha,1\right),
\]
valid for all $x\in(0,\pi)$ and $\alpha\in(0,1)$.  Taking the supremum over $x\in\mathbb{S}$
now yields $c_\alpha^2\lesssim\max(c_\alpha,1)$, hence $c_\alpha\lesssim 1$ uniformly in $\alpha\in(0,1)$.
With this uniform bound, we may now take $ \alpha\nearrow 1$ in \eqref{eq:calpha} to get
\[
\gamma-\varphi(x)\lesssim x\left(1+|\log x|\right),
\]
valid for all $x\in(0,\pi)$.  Taking $x\to 0^+$ {finally} yields desired upper bound in \eqref{eq:reg}.
\end{proof}

Theorem \ref{thm:reg} implies that \emph{if} there exists an even, $2\pi$-periodic continuous 
solution of \eqref{profile}
that is nondecreasing on $(0,\pi)$ with $\varphi(0)=\gamma$, then the derivative of such a solution
blows up at $x=0$ at a logarithmic rate, i.e., the solution is \emph{logarithmically cusped}.  This is in contrast
to analogous highest solutions of the unidirectional Whitham equation, where such solutions
are {cusped} with their derivatives at the highest point blowing up at an algebraic rate \cite{EWhighest}.

From the proof of Theorem \ref{thm:reg} and the mapping property \eqref{eq:mapping}, it is tempting to
expect that such a highest solution $ \varphi$ of \eqref{profile} may belong to $C^1_*(\mathbb{S})$.  However,
this is readily seen to be false due to the following characterization: a function $f\in C^0(\mathbb{S})$
belongs to the Zygmund space $f\in C^1_*(\mathbb{S})$ if and only if
\[
\sup_{h\neq 0}\frac{\|f(\cdot+h)+f(\cdot-h)-2f\|_{L^\infty(\mathbb{S})}}{|h|}<\infty.
\]
Indeed, taking $f(x)=|x\log|x||$ we find that
\[
\|f(\cdot+h)-f(\cdot-h)-2f\|_{L^\infty(\mathbb{S})}\geq|f(h)+f(-h)-2f(0)|=2|h\log|h||,
\]
which immediately implies that $f\notin C^1_*(\mathbb{S})$ by the above characterization.
The precise regularity of such a highest wave of \eqref{profile} is thus an interesting question:
it is neither continuously differentiable, nor Lipschitz, nor does it belong to the Zygmund
space $C^1_*(\mathbb{S})$, {nor does its  derivative belong to the space of functions of bounded mean oscillation \({\rm BMO}(\RM)\), due to its sign-changing property at the singular crest. We settle here for the asymptotic property \eqref{eq:reg}, but it is reasonable to believe that the optimal regularity for the highest wave belongs is a dyadic space (see, e.g., \cite{MR2651916}).}

It now remains to prove that  a symmetric, $2\pi$-periodic, 
uni-modal solution of \eqref{profile} exists with $\varphi(0)=\gamma$.
This is the goal of the next section.

\section{Bifurcation of smooth periodic waves with a single crest in each period}\label{sec:bifurcation}
In this section we develop a bifurcation theory for the  profile equation \eqref{profile}.  
We begin with local bifurcation theory, and then extend
this to a global theory, carefully characterizing the end of the bifurcation branch with the help of nonlocal arguments (non-local here referring to \(x\)-space).  In particular, we will
see that the bifurcation branch terminates in a highest wave that is even, monotone increasing on $(-\pi,0)$,
and satisfies $\varphi(0)=\gamma$.  By Theorem \ref{thm:reg} above, it follows that this highest wave will
be a cusped traveling wave solution of the {bidirectional}  Whitham model.

Our study starts with the existence of small-amplitude, nonlinear solutions to \eqref{profile}.  To this end, notice that
for $\alpha\in(\frac{1}{2},1)$ the space $C^\alpha_{\rm even}(\mathbb{S})=C^{\alpha}_{*,{\rm even}}(\mathbb{S})$ of even functions in $C^\alpha(\mathbb{S})$ forms a Banach algebra.
Further, since $K$ is even it follows that, for such $\alpha$, $C^\alpha_{\rm even}(\mathbb{S})$ is {invariant under the action of $\mathcal{K}$, in view of that
\[
\mathcal{K}:C^\alpha_{\rm even}(\mathbb{S})\to C^{\alpha+1}_{\rm even}(\mathbb{S})\hookrightarrow C^\alpha_{\rm even}(\mathbb{S}).
\]
}
Seeking solutions of \eqref{profile} in $C^\alpha_{\rm even}(\mathbb{S})$, we consider the function
\[
F:C^{k,\alpha}_{\rm even}(\mathbb{S})\times\mathbb{R}\to C^{k,\alpha}_{\rm even}(\mathbb{S})
\]
defined by
\[
F(\varphi,c):=\mathcal{K}\varphi-N(\varphi;c),
\]
with $N$ as in~\eqref{profile}. Then $F$ is a real-analytic operator.
Observing that $F(0,c)=0$ for all $c\in\mathbb{R}$ and that
\[
\partial_uF[(0,c)]=\mathcal{K}-c^2
\]
{one readily sees} that $\ker\left(\partial_uF[(0,c)]\right)$ is trivial unless 
\begin{equation}\label{eq:dispersion relation}
c^2=\frac{\tanh(k)}{k}
\end{equation}
for some $k\in\NM_0$, in which case $\ker\left(\partial_uF[(0,c)]\right)={\rm span}\{\cos(kx)\}$.  In particular, such { values of} $c^2$ are simple eigenvalues of 
$\mathcal{K}$.  It follows from the analytic version of the Crandall--Rabinowitz theorem \cite{BTbook} that near any {point $(0,c)$ for such values of \(c\) there exists small-amplitude solutions \((\varphi,c)\) in $C^{k,\alpha}_{\rm even}(\mathbb{S}) \times \RM$} of the nonlocal profile equation \eqref{profile} (for details, see, e.g., \cite{EK11}). This result is summarized in the following proposition.  The asymptotic formulas could be obtained as in \cite{EK11} by using the general theory from \cite{MR2004250}. Or, as we have an analytic curve in \(C^\alpha(\mathbb{S}) \times \mathbb{R}\), \(\alpha > \frac{1}{2}\) (for which \(\varphi \in C^\alpha(\mathbb{S})\) has absolutely convergent Fourier series), by means of direct expansions in the equation: see, for example,
\cite{J13}.

\begin{proposition}\label{prop:local}
Fix $\alpha\in(\frac{1}{2},1)$.  For each integer $k\geq 1$ there exists $c_k:=\sqrt{\tanh(k)/k}$ and a local, analytic curve
\[
{\tau}\mapsto\left(\varphi({\tau}),c({\tau})\right)\in C^\alpha_{\rm even}(\mathbb{S})\times(0,1),
\]
defined in a (real) neighborhood of ${\tau}=0$, 
of non-trivial $2\pi$/k-periodic solutions of profile equation \eqref{profile} that bifurcate from the trivial solution curve 
$c\mapsto (0,c)$ at $(\varphi(0),c(0))=(0,c_k)$.  For $|{\tau}|\ll 1$, we have
\begin{align*}
&\varphi(x;\tau )=\tau \cos(kx)+\frac{3c_k\tau ^2}{4}\left(\frac{1}{c_k^2-1}+\frac{\cos(2kx)}{c_k^2-c_{2k}^2}\right)+  \mathcal{O}(\tau ^3)
\end{align*}
and
\begin{align*}
&c(\tau )=c_k+\frac{3\tau ^2}{8}\left[-\frac{1}{2c_k}+3c_k\left(\frac{1}{c_k^2-1}+\frac{1}{2(c_k^2-c_{2k}^2)}\right)\right]+  \mathcal{O}(\tau ^4).
\end{align*}
In a neighborhood of the bifurcation point $(0,c_k)$ these comprise
all non-trivial solutions of $F(\varphi,c)=0$ in $C^\alpha_{\rm even}(\mathbb{S})\times\mathbb{R}$, and there are no other
bifurcation points $c>0$, $c\neq 1$, in $C^\alpha_{\rm even}(\mathbb{S})$.  
\end{proposition}

There are a couple of comments to be given in connection to Proposition~\ref{prop:local}. We list these together in the following remark.

\begin{remark}\label{rem:local}
$ $\\
(i) When $k=0$, meaning $c=1$, a bifurcating line $c\mapsto(\frac{3c-\sqrt{8+c^2}}{2},c)$ of constant
solutions intersects the trivial solution curve $c\mapsto (0,c)$, resulting in a transcritical bifurcation.  Together, these constitute all solutions in $C^\alpha_{\rm even}(\mathbb{S})$ in a neighborhood of $(\varphi,c)=(0,1)$.\\[6pt]
(ii) There is nothing particular about the choice of  \(\alpha\) in this section. In fact, as we shall show, all small enough solutions are smooth, and so all agree by uniqueness. The choice \(\alpha > \frac{1}{2}\) is by convenience. For the global argument, however, the choice of \(\alpha\) has some implications for the proof. In particular, \(\alpha < 1\) makes it easier to rule out one alternative in Theorem \ref{thm:global} along the curve of solutions. \\[6pt]
(iii) In contrast to the corresponding unidirectional Whitham equation, we see from \eqref{eq:dispersion relation} that the {bidirectional} equation \eqref{profile} admits two families of small-amplitude solutions; one for positive values of \(c\), and one for negative. The change of variables \((\varphi,c) \mapsto -(\varphi,c)\) in \eqref{profile}
however guarantees that these are in one-to-one correspondence to each other.
\end{remark}

We shall now analyze the global structure of the local bifurcation curve constructed in Proposition~\ref{prop:local}.  To begin with, 
we introduce the admissible set 
\[
U:=\left\{(\varphi,c)\in C^\alpha_{\rm even}(\mathbb{S})\times \RM:\varphi<\gamma \right\},
\]
where \(\gamma = c (1 - \frac{1}{\sqrt{3}})\), and the set of solutions
\[
S:=\left\{(\varphi,c)\in U:F(\varphi,c)=0\right\}.
\]
With these definitions, the following global bifurcation result is an easy adaptation of \cite[Theorem 9.1]{BTbook}.

\begin{theorem}\label{thm:global}
The local bifurcation curve $\tau \mapsto (\varphi(\tau ),c(\tau ))\in C^\alpha_{\rm even}(\mathbb{S})\times \RM$ of solutions
of \eqref{profile} constructed in Proposition~\ref{prop:local} for \(k=1\) extend to a global continuous curve
of solutions
\[
\mathcal{R}=\left\{(\varphi(\tau ),c(\tau )):\tau \in[0,\infty)\right\}\subset S,
\]
that allows a local real-analytic reparametrization  about each $\tau >0$.  Furthermore, one of the following
alternatives hold:
\begin{itemize}
\item[(i)] $\|(\varphi(\tau ),c(\tau ))\|_{C^\alpha(\mathbb{S})\times\mathbb{R}}\to\infty$ as $\tau \to\infty$.\\[-6pt]
\item[(ii)] ${\rm dist}\left(\mathcal{R},\partial U\right)=0$.\\[-6pt]
\item[(iii)] The function $\tau \mapsto(\varphi(\tau ),c(\tau ))$ is $T$-periodic for some finite $T>0$.
\end{itemize}
\end{theorem}

\begin{proof}
By \cite[Theorem 9.1]{BTbook}, it is enough to verify that closed and bounded subsets of $S$ are compact
in $C^\alpha_{\rm even}(\mathbb{S})\times(0,1)$, and that $c(\tau )$ is not identically constant for $0<\tau \ll 1$.
As for the former, let $V\subset S$ be a closed and bounded set in $S$. By closedness there exists \(\delta > 0\) such that
\[
\inf_{V} \{\gamma -  \max \varphi\} > \delta, 
\]
for all pairs \((\varphi,c)\) in \(V\). One therefore has, similarly to \eqref{eq:nemytskii}, that the function
\[
N^{-1}(\cdot, c)
\] 
is a well-defined and smooth on \(V\). Since ${\mathcal K}$ maps \(C^\alpha_{\rm even}(\mathbb{S})\) continuously into  \(C^{\alpha+1}_{\rm even}(\mathbb{S})\),  and 
$C^{\alpha+1}_{\rm even}(\mathbb{S})$ is compactly embedded in $C^\alpha_{\rm even}(\mathbb{S})$, it follows that the composition
$N^{-1}(\cdot, c) \circ \mathcal{K}$ maps bounded sets into precompact sets. But for solutions in \(S\) we have
\[
\varphi = \mathcal [N^{-1}(\cdot, c) \circ \mathcal{K}]\varphi,
\]
so that \(V\) is indeed precompact. Since $V$ is also closed, it is compact.

Now, for $k=1$ in Proposition~\ref{prop:local} one has $c''(0)<0$, whence the {result is a consequence of} the global bifurcation result in \cite{BTbook}.
\end{proof}

We  proceed to classify the limiting behavior of the solutions at the end of the bifurcation curve. We will prove 
that both alternatives (i) and (iii) in Theorem \ref{thm:global} are excluded, whence (ii) happens by the curve approaching a ``highest'' wave, which we shall show satisfies \(\varphi(0) = \gamma\). As a first step, observe that Proposition \ref{prop:local} implies $0<c(\tau )<1$ for all $\tau >0$ sufficiently
small.  The next result shows, in combination with Remark~\ref{rem:sign-changing}, that the global bifurcation curve cannot pass \(c=0\) or \(c = 1\) without crossing the curve of zero solutions.

\begin{lemma}\label{lemma:uniqueness}
For \(c=1\) the zero solution is the unique solution of \eqref{profile} satisfying $\varphi\leq\gamma$; for \(c=0\) one additionally has {the solution} \(\varphi \equiv -\sqrt{2}\).
\end{lemma}

\begin{proof}
When $c=1$,  $\varphi$ solves the equation
\[
\left(1-\mathcal{K}\right)\varphi=\frac{1}{2}\varphi^2\left(3c-\varphi\right).
\]
Recalling that \(\hat K(0) = 1\), we find by integrating over $(-\pi,\pi)$  that
\[
\int_{-\pi}^\pi \varphi^2\left(3c-\varphi\right)dx=0.
\]
Since $\varphi \leq \gamma < 3c$ it follows that $\varphi\equiv 0$, as claimed.

Now, let $\varphi$ be a solution of \eqref{profile} with $c=0$ satisfying $\max_x \varphi(x)\leq \gamma = 0$. 
Since Lemma \ref{lemma:Kp} implies that $\mathcal{K}\in\mathcal{L}(L^\infty(\mathbb{S})$ with unit operator norm,
the profile equation \eqref{profile} implies that
\[
{\textstyle \frac{1}{2}}(\min_x \varphi(x))^{3}\geq \min_{x}\varphi(x) \quad\text{and }\quad {\textstyle \frac{1}{2}}(\max_x \varphi(x))^3\leq\max_{x}\varphi(x).
\]
Either \(\max_x \varphi(x) = 0\), in which case it follows from Remark~\ref{rem:sign-changing} that \(\varphi \equiv 0\), or \(\max_x \varphi(x) < 0\) with \((\max_x \varphi(x))^2 \geq 2\).
Since $(\min_x \varphi(x))^2 \leq 2$, we must have \(\varphi \equiv -\sqrt{2}\) in the second case (else \(\max \varphi\) would be strictly smaller than \(\min \varphi\)).
\end{proof}

We now show that alternative (iii) in Theorem \ref{thm:global} is excluded.  {Consider} the set
\begin{equation}\label{eq:cone}
\Lambda = \{\varphi \in C^\alpha(\mathbb{S}) \colon \varphi \text{ is even and non-decreasing on } (-\pi,0)\},
\end{equation}
which is a closed cone in \(C^\alpha(\mathbb{S})\).  Observe that from Proposition \ref{prop:local}
and Lemma \ref{lemma:uniform_c-bound}, we have that the solutions $ \varphi(\tau )$ for 
all $0<\tau \ll 1$, and may be expanded as 
\[
{\varphi}(\tau )=\tau \cos+\mathcal{O}(\tau ^2)
\]
in $C^\alpha(\mathbb{S})$. As \(\tau  \mapsto \varphi(\tau )\) analytic, and the identity map \(\varphi \mapsto \varphi\) given in \eqref{eq:nemytskii} is smooth \(C^s_*(\mathbb{S}) \to C^{s+1}_*(\mathbb{S})\), for all \(s \geq 0\) and all solutions satisfying \(\varphi < \gamma\), it follows in particular that \(\tau  \mapsto \varphi(\tau )\) is smooth \((-\delta, \delta) \to C^{2}(\mathbb{S})\) around a neighbourhood of the origin. Taylor's formula and uniqueness in the larger space \(C^\alpha(\mathbb{S})\) then implies that the above asymptotics hold in \(C^2(\mathbb{S})\). It is then easy to check that  ${\varphi}(\tau )\in\Lambda\setminus\{0\}$ for all $0<\tau \ll 1$.  
The next result
shows that all solutions \(\left\{\varphi(\tau )\right\}_{\tau >0}\) on the global bifurcation curve \(\mathcal R\) belong to \(\Lambda \setminus \{0\}\), and consequently that alternative (iii) in Theorem \ref{thm:global} is excluded.

\begin{lemma}\label{lemma:noperiod}
One has \(\varphi(\tau ) \in \Lambda \setminus \{0\}\) for all \(\tau  > 0\) along the global bifurcation curve. In particular, alternative (iii) in Theorem~\ref{thm:global} cannot occur.
\end{lemma}

{
\begin{remark}
The proof of Lemma~\ref{lemma:noperiod} requires the nodal pattern of \(\varphi(\tau )\) proved in Theorem~\ref{thm:nodal}, combined with the properties of the  curve \(\Gamma_-\) of trivial solutions. A detailed proof, based on the theory in \cite{BTbook}, is carried out in \cite{EWhighest}.  However, in \cite{EWhighest} there exists
a Galilean symmetry relating solutions with wave speed $c>1$ to those with wave speed $0\leq c<1$, and
this fact was used in the proof.  Since the profile equation \eqref{profile} in the present
case does \emph{not} admit such a symmetry, for completeness we outline the main steps of the proof, providing
full details only in the necessary modifications.
\end{remark}
}

\begin{proof}
To begin, notice that if the lemma were false then the number
\[
\bar\tau :=\sup\left\{\eta>0 \colon {\varphi} (\eta)\in\Lambda\setminus\{0\}\right\}
\]
would be finite and strictly positive.  Since $\Lambda$ is a closed subset of $C^\alpha(\mathbb{S})$
it {must be} that ${\varphi}(\bar\tau )\in\Lambda$.  As in \cite{EWhighest}, one may use the nodal
pattern in Theorem \ref{thm:nodal} to argue  that
if $ \varphi(\bar\tau )$ is non-constant, then it is an interior point of $\mathcal{R}\cap\Lambda$
with respect to the $C^\alpha(\mathbb{S})$ metric relative to $\mathbb{R}$, contradicting
the definition of $\bar\tau $.  Therefore, by the discussion
at the beginning of Section \ref{sec:a priori}, one of the following must be true:
\[
 \varphi (\bar\tau )\in\Gamma_-,\quad \varphi(\bar\tau )\in\Gamma_+,\quad\textrm{or}
\quad\varphi(\bar\tau )\equiv 0.
\]
Since $\mathcal{R}\subset S$, the possibility that $ \varphi(\bar\tau )\in\Gamma_+$ is clearly excluded.
Further, if $c(\bar\tau )>1$ then there 
exists a $0<\tau _1<\bar\tau $ such that $c(\tau _1)=1$.  In this case, {Lemma} \ref{lemma:uniqueness} implies
that $ \varphi(\tau _1)\equiv 0$, which contradicts the definition of $\bar\tau $.  Thus, it must be that
$0\leq c(\bar\tau )\leq 1$.  Suppose now that $ \varphi(\bar\tau )\in\Gamma_-$ with $c(\bar\tau )\in[0,1]$.
If $c(\bar\tau )\in[0,1)$ then since  $\max_x\left[\Gamma_-(c(\bar\tau ))\right](x)<0$ it follows by 
continuity that there {exists an $\tau _2 \in (0,\bar\tau )$} such that $\max_x  \varphi(\tau _2)=0$ which, 
by Remark \ref{rem:sign-changing},
implies that ${\varphi}(\tau _2)\equiv 0$, again contradicting the definition of $\bar\tau $.
Consequently, if ${\varphi}(\bar\tau )\in\Gamma_-$ it must be that $c(\bar\tau )=1$ and hence ${\varphi}(\bar\tau )\equiv 0$
by Lemma \ref{lemma:uniqueness}.  
Recalling that the only solutions with $0<c<1$ that connects to $  ({\varphi}, c)=(0,1)$ are the constant
solutions, it follows that either ${\varphi}(\tau )=0$ for all $0<\bar\tau -\tau \ll 1$, 
again contradicting the definition of $\bar\tau $,
or ${\varphi}(\tau )\in \Gamma_-$ with $0<c(\tau )<1$
for all $0<\bar\tau -\tau \ll 1$.  The latter case, however, has already been excluded.

{In summary,  $\varphi(\bar\tau )\equiv 0$ and $c(\tau )\in(0,1)$} uniformly for all $\tau \in[0,\bar\tau ]$.  
The remainder of the proof {goes} as in \cite{EWhighest}, hence we just outline the details.
By Proposition \ref{prop:local} we see that $({\varphi}(\bar\tau ),c(\bar\tau ))$ is a local bifurcation point
and, since $\cos(kx)\in\Lambda$ if and only if $k=1$, we find that $\mathcal{R}$ coincides with
the primary branch, i.e., with itself, for $0<\bar\tau -\tau \ll 1$ (here we think of \(\mathcal R\) as parametrized by \(\tau\)).  Precisely, there exists a countably
infinite set of pairs $\{(\tau _{1,j},\tau _{2,j})\}_{j=1}^\infty$ such that $\tau _{1,j}\searrow 0$ and
$\tau _{2,j}\nearrow \bar\tau $ as $j\to\infty$ with $\mathcal{R}(\tau _{1,j})=\mathcal{R}(\tau _{2,j})$
for all $j\geq 1$. It follows that the kernel of $DF(\mathcal{R}(\tau _{1,j}))$ is nontrivial, having
dimension at least one for each $j\geq 1$.  Since the values of $\tau $ for which the kernel
of $DF(\mathcal{R}(\tau ))$ is nontrivial are known to be isolated, c.f. \cite{BTbook}, this yields
a contradiction.  {Therefore,} such an $\bar\tau $ does not exist, and ${\varphi}(\tau )\in\Lambda\setminus\{0\}$
for all $\tau >0$ as claimed.
\end{proof}

Combining Theorem \ref{thm:nodal} with Lemma \ref{lemma:noperiod}, {one obtains} that all solutions ${\varphi}(\tau )$
along the main bifurcation curve $\mathcal{R}$ are nontrivial, smooth, even, and strictly increasing on $(-\pi,0)$.
We now wish to pass to the limit $\tau \to\infty$ along the global bifurcation branch, obtaining
a nontrivial highest wave.  To this end, the next result shows that 
the solution set is relatively compact in the appropriate space.

{
\begin{lemma}\label{lemma:arzela}
Given a sequence  \(\{(c(\tau _n),\varphi(\tau _n))\}_n\) of solution pairs along the global bifurcation curve from Theorem~\ref{thm:global} satisfying \( c_n \lesssim 1\),  there exists a subsequence converging  uniformly to a solution pair \((\varphi,c)  \in C(\mathbb{S}) \times \RM\).
\end{lemma}

\begin{proof}
Let \(\{(c_n,\varphi_n)\}_n = \{(c(\tau _n),\varphi(\tau _n))\}_n\).  By Lemmas~\ref{lemma:uniqueness} and~\ref{lemma:noperiod} we have \(c_n > 0\) for each \(n\). By assumption, there thus exists a subsequence such that \(c_n \to c\) in \(\RM\).  
Theorem~\ref{thm:reg}(ii) then guarantees  that \(\{\varphi_n\}_n\) is uniformly bounded in \(C^\alpha(\mathbb{S})\), where we may pick \(\alpha \in (0,1)\) by convenience. 
In particular, \(\{\varphi_n\}_n\) is an equicontinuous family of solutions, and thus, by Arzela--Ascoli's lemma and compactness of \(\mathbb{S}\), has a subsequence converging uniformly to a function \(\varphi \in C(\mathbb{S})\). To see that \((\varphi,c)\) solves \eqref{profile}, it suffices to note that \(K_p \in L^1(\mathbb{S})\), and that \(N\) is smooth in \(\varphi\) and \(c\). We may thus let \((\varphi_n,c_n) \to (\varphi,c)\) in \(C(\mathbb{S}) \times \RM\) in \eqref{profile} to conclude that \( (\varphi,c)\) solves the same equation.
\end{proof}
}

Finally, to exclude that we end up with a trivial wave, we prove the following a priori bound 
on the wavespeed along the main bifurcation branch.

\begin{lemma}\label{lemma:c_uniformbound}
Along the global bifurcation curve in Theorem \ref{thm:global}, the wavespeed \(c(\tau )\) satisfies
\[
 0<c_{\rm min} \leq c(\tau ) \leq c_{\rm max} <1.
\]
In particular, alternative (i) in Theorem \ref{thm:global} is excluded.
\end{lemma}

\begin{proof}
From the proof of Lemma \ref{lemma:noperiod}, we find that $c(\tau )<1$ uniformly for all $\tau >0$.  Indeed,
Proposition \ref{prop:local} implies $c(\tau )<1$ for $\tau >0$ sufficiently small, and Lemma \ref{lemma:uniqueness}
implies that the only {way} that $c(\tau )$ can approach $c=1$ along $\mathcal{R}$ is for the solutions
${\varphi}(\tau )$ to approach the trivial solution ${\varphi}=0$.  However, Proposition \ref{prop:local} implies
that the unique solutions in a neighborhood of $({\varphi},c)=(0,1)$ are the constant solutions.  
Since the proof of Lemma \ref{lemma:noperiod} shows that the main bifurcation curve $\mathcal{R}$ does
not connect to the two lines of constant solutions, it follows that $c(\tau )$ is uniformly bounded {away from} $c=1$ for all $\tau >0$.  

To verify that $c(\tau )\gtrsim 1$ uniformly,
suppose, on the contrary, that there exists a sequence 
$(\varphi_n,c_n) = (\varphi(\tau _n),c(\tau _n)) \in S\cap(\Lambda\times\mathbb{R})$
such that $c_n\searrow 0$ as $n\to\infty$.  By Lemma~\ref{lemma:arzela}, we know that $\varphi_n\to\varphi_0$
in $C(\mathbb{S})$ for some solution $\varphi_0 \leq \gamma$, non-decreasing on \((-\pi,0)\). Since $c=0$, \(\varphi_0\) must be constant by Lemma~\ref{lemma:uniqueness}. There are now two cases. If $\varphi_0 =  0$ one immediately gets a contradiction from 
Lemma \ref{lemma:lower_blowup improved}.
If, on the other hand,  \(\varphi_0 = - \sqrt{2}\), the solution curve must, by continuity, have passed a point \((c_*,\varphi_*)\) at which \(c_* > 0\) and \(\max_x \varphi_*(x) = 0\).  (The asymptotic formula for  \(\varphi(\tau )\) in Proposition~\ref{prop:local} guarantees that \(\max_x \varphi(\tau ) > 0\) for small \(\tau  > 0\).) According to Remark~\ref{rem:sign-changing},  \(\varphi_* \equiv 0\), which contradicts the fact that \(\varphi(\tau )\) is strictly monotone on a half-period by Lemma~\ref{lemma:noperiod}, for any \(\tau  \in (0,\infty)\).

Finally, since $c(\tau ) \eqsim 1$ it follows from Theorem~\ref{thm:reg}(ii) that the quantity  $\|({\varphi}(\tau ),c(\tau ))\|_{C^\alpha(\mathbb{S})\times\mathbb{R}}$  is uniformly bounded
for all $\tau>0$; recall here that \(\alpha \in (\frac{1}{2},1)\) is fixed.  This excludes alternative (i) in Theorem \ref{thm:global}, as claimed.

\end{proof}

{Combining the  results of this section}, we obtain the following:
By Proposition \ref{thm:global}, Lemma~\ref{lemma:noperiod} and Lemma~\ref{lemma:c_uniformbound} 
there exists a sequence  \(\{(\varphi_n,c_n)\}_n = \{\varphi(\tau _n),c(\tau _n)\}_n\) 
that approaches the boundary of \(U\) as \(n \to \infty\), meaning that 
\[
\lim_{n \to \infty}(\gamma_n - \max \varphi_n) = 0, 
\]
with \(\gamma_n = (1- \frac{1}{\sqrt{3}})c_n\). Lemma~\ref{lemma:c_uniformbound} guarantees that \(\inf_n c_n > 0\), and it furthermore follows from Lemma~\ref{lemma:uniqueness} and Lemma~\ref{lemma:noperiod} that \(\sup_n c_n {<} 1\). Hence, Lemma~\ref{lemma:arzela} yields the existence of a convergent subsequence, denoted again by \(\{\varphi_n,c_n\}_n\), converging in \(C(\mathbb{S}) \times \RM\) to a solution pair \((\varphi,c)\). The solution \(\varphi\) is non-decreasing on \([-\pi,0]\) and satisfies \(\varphi(0) = \gamma\), with \(\gamma = \lim_{n\to \infty} \gamma_n\). 
Theorem \ref{thm:reg} now immediately yields the following result.

{
\begin{theorem}\label{thm:main}
In Theorem \ref{thm:global}, only alternative (ii) occurs. Given any sequence of positive numbers $\tau _n$ 
with $\tau _n\nearrow\infty$, there exists a limiting wave ${\varphi}$ obtained as the uniform limit
of a subsequence $\{\varphi(\tau _{n_k})\}_k$.  The limiting wave is a solution of \eqref{profile}
with $c=\lim_{k\to\infty}c(\tau _{n_k})$ and is even, $2\pi$-periodic, and satisfies 
${\varphi}(0)=\left(1-{\textstyle\frac{1}{\sqrt{3}}}\right)c=:\gamma$.  
Further, it is strictly increasing on $(-\pi,0)$,
smooth on $\mathbb{R}\setminus2\pi\mathbb{Z}$, and satisfies
\[
\gamma-  \varphi(x)\simeq \left|x\log|x|\right|,
\]
for all $|x|\ll 1$ sufficiently small.
\end{theorem}
}

\begin{remark}\label{rem_positivehighest}
As mentioned in Remark \ref{signdefinite}, the cusped traveling wave solution $\varphi$ in Theorem \ref{thm:main} is necessarily sign-changing.  The existence of 
a highest cusped wave that is sign-definite was suggested numerically to exist in the recent work \cite{CJ17}.  We note that the theory presented in the previous sections may be used to prove the existence of a positive, logarithmically-cusped, highest wave, occurring at the end of the global bifurcating branch continuing from the
trivial solution $\Gamma_-$ with $c>1$ and terminating at $\max(\varphi)=\gamma$.  
This extension is outlined in Appendix \ref{appendix_b} below.
\end{remark}

\appendix

\section{Numerical Schemes}\label{Appendix}

In this appendix, we turn to the numerical approximation of solutions of \eqref{profile}.
To numerically approximate the even, $2\pi$-periodic solutions of \eqref{profile}, we employ a cosine-collocation method as discussed in \cite[Section 5]{EK11} in conjunction with the pseudo-arclength continuation method to achieve a curve of solutions bifurcating from the solution of trivial amplitude.  Here, we discuss some of the details of these methods.

\subsection{Cosine Collocation Method}

Solutions $\varphi$ of \eqref{profile} that are even $2\pi$-periodic may be naturally expanded 
in a Fourier cosine basis as
\begin{equation}  \label{E:fourier_series}
\varphi(x) = \sum_{n=0}^\infty \widehat{\varphi}(n)\cos(nx),
\end{equation}
where
\begin{equation}  \label{E:fourier_coefficients}
\widehat{\varphi}(n) = \begin{cases}
\displaystyle \frac{1}{2\pi} \int_{-\pi}^\pi \varphi(x) \, dx & \text{if $n = 0$} \\
\displaystyle \frac{1}{\pi} \int_{-\pi}^\pi \varphi(x)\cos(nx) \, dx & \text{if $n \geq 1$}.
\end{cases}
\end{equation}
To approximate $\varphi$, we first truncate the Fourier series \eqref{E:fourier_series} to $N\in\NM$ terms:
\[
\varphi_N(x) := \sum_{n=0}^{N-1} \widehat{\varphi}(n) \cos(n x).
\]
Recalling that $\varphi$ is to be even, for each $n=0,1,\ldots N-1$ the Fourier coefficients 
$\widehat\varphi(n)$ may be approximated by discretizing $[0,\pi]$ into $N+1$ subintervals
\[
0 < x_1 < x_2 < \cdots < x_{N-1} < x_N < \pi,
\]
where $x_m = \dfrac{(2m-1)\pi}{2N}$ for $m=1,2,\ldots,N$ are the so-called \emph{collocation points},
and applying midpoint quadrature.  This gives the approximation
\begin{equation}  \label{E:fourier_coefficients_approx}
\widehat{\varphi}(n) \approx \widehat{\varphi}_N(n) := w(n) \sum_{m=1}^N \varphi_N(x_m)\cos(n x_m),
\end{equation}
with
\[
w(n) = \begin{cases}
1/N & \text{if $n=0$} \\
2/N & \text{if $n=1,2,\ldots,N-1$}
\end{cases}
\]
yielding the discrete cosine representation of $\varphi_N(x)$:
\begin{align*}
\varphi_N(x) &= \sum_{m=1}^N \left( \sum_{n=0}^{N-1} w(n)\cos(n x_m)\cos(n x) \right) \varphi_N(x_m).
\end{align*}
Moreover, the action of $\mathcal{K}$ on $\varphi$ may be approximated via
\begin{align*}
\mathcal{K}\varphi(x) \approx \mathcal{K}_N \varphi_N(x) &:= \sum_{n=0}^{N-1} \widehat{\mathcal{K}\varphi_N}(n)\cos(n x) \\
&= \sum_{m=1}^N \left( \sum_{n=0}^{N-1} \frac{\tanh(n)}{n} \, w(n)\cos(n x_m)\cos(n x) \right) \varphi_N(x_m).
\end{align*}

Enforcing the profile equation \eqref{eq:profile'} at each of the collocation points $x_i$ for $i=1,\ldots,N$ and replacing $\varphi$ and $\mathcal{K}\varphi$ with their respective approximations $\varphi_N$ and $\mathcal{K}_N \varphi_N$, we have
\begin{equation*}
\varphi_N(x_i) \left( c-\frac{1}{2}\varphi_N(x_i) \right) \left( c-\varphi_N(x_i) \right) - \mathcal{K}_N \varphi_N(x_i) = 0,
\end{equation*}
a nonlinear system of $N$ equations in the $N+1$ unknowns $c$, $\varphi_N(x_1)$, $\varphi_N(x_2), \ldots, \varphi_N(x_N)$.  For convenience, let
\[
\varphi_N^i := \varphi_N(x_i), \quad \text{and} \quad \mathcal{K}_N \varphi_N^i := \mathcal{K}_N \varphi_N(x_i),
\]
and for $y := (c,\varphi_N^1,\varphi_N^2,\ldots,\varphi_N^N) \in \RM^{N+1}$, define $f(y) : \RM^{N+1} \to \RM^N$ by
\begin{equation}  \label{E:pseudoarclength_f}
f(y) := (f_1(y), \ldots, f_N(y)), \quad f_i(y) := \varphi_N^i \left( c-\frac{1}{2}\varphi_N^i \right) \left( c-\varphi_N^i \right) - \mathcal{K}_N \varphi_N^i.
\end{equation}
Using the local bifurcation formulas in Proposition \ref{prop:local} and a numerical continuation algorithm, we will solve the nonlinear equation
\[
f(y) = 0
\]
to obtain wavespeeds $c$ and points $\varphi_N^i = \varphi_N(x_i)$ on the corresponding approximate profile wave.

\subsection{Continuation by the Pseudo-Arclength Method}

For a given $f : \RM^{N+1} \to \RM^N$, consider the problem of finding points on the curve defined by $f(y) = 0$.  Given a point $y_0 \in \RM^{N+1}$ such that $f(y_0) = 0$ and an initial unit tangent direction $z_0 \in \RM^{N+1}$, one can find another point $y_1$ such that $f(y_1) = 0$ via a predictor-corrector method known as the \emph{pseudo-arclength} method.  This method is outlined in the following three steps:
\\

\noindent 
(1) For a \emph{step size} $h$, extrapolate from $y_0$ along the tangent direction $z_0$ to form the predictor $y_0^p := y_0 + hz_0$.
\\

\noindent
(2) From $y_0^p$, correct the extrapolation by projecting onto the curve $f(y) = 0$ in the direction orthogonal to $z_0$.  That is, solve for $y_1 \in \RM^{N+1}$ in the following nonlinear system of $N+1$ equations in $N+1$ unknowns:
\[
\left\{
\begin{aligned}
&f(y_1) = 0 \\
&z_0 \cdot (y_1 - y_0^p) = 0.
\end{aligned}
\right.
\]
This step may be accomplished, for example, by Newton's method.
\\

\noindent 
(3) Obtain a suitable tangent direction $z_1$ at $y_1$ by solving for $z_1$ in the following system of $N+1$ equations and $N+1$ unknowns:
\begin{empheq}[left=\empheqlbrace]{align}
& Df(y_1)z_1 = 0  \label{E:tangency_condition} \\
& z_0 \cdot z_1 = 1.  \label{E:orientation_condition}
\end{empheq}
Here, the first equation ensures $z_1$ is tangent to $f(y)=0$ at $y_1$, while the second equation guarantees
the angle between $z_0$ and $z_1$ is acute, ensuring a consistent orientation of the tangent vectors. Note that resulting $z_1$ above will not be of unit length and should be normalized before iterating the method.

\

The above algorithm can be iterated to continue from a point $y_k$ on the curve to another point $y_{k+1}$ such that $f(y_{k+1})=0$.  See Figure \ref{F:pseudoarclength} for a graphical illustration of the method.  

\begin{figure}[t] 
\centering
\includegraphics[scale=0.4,clip=true,trim=2cm 2cm 1.5cm 2cm]{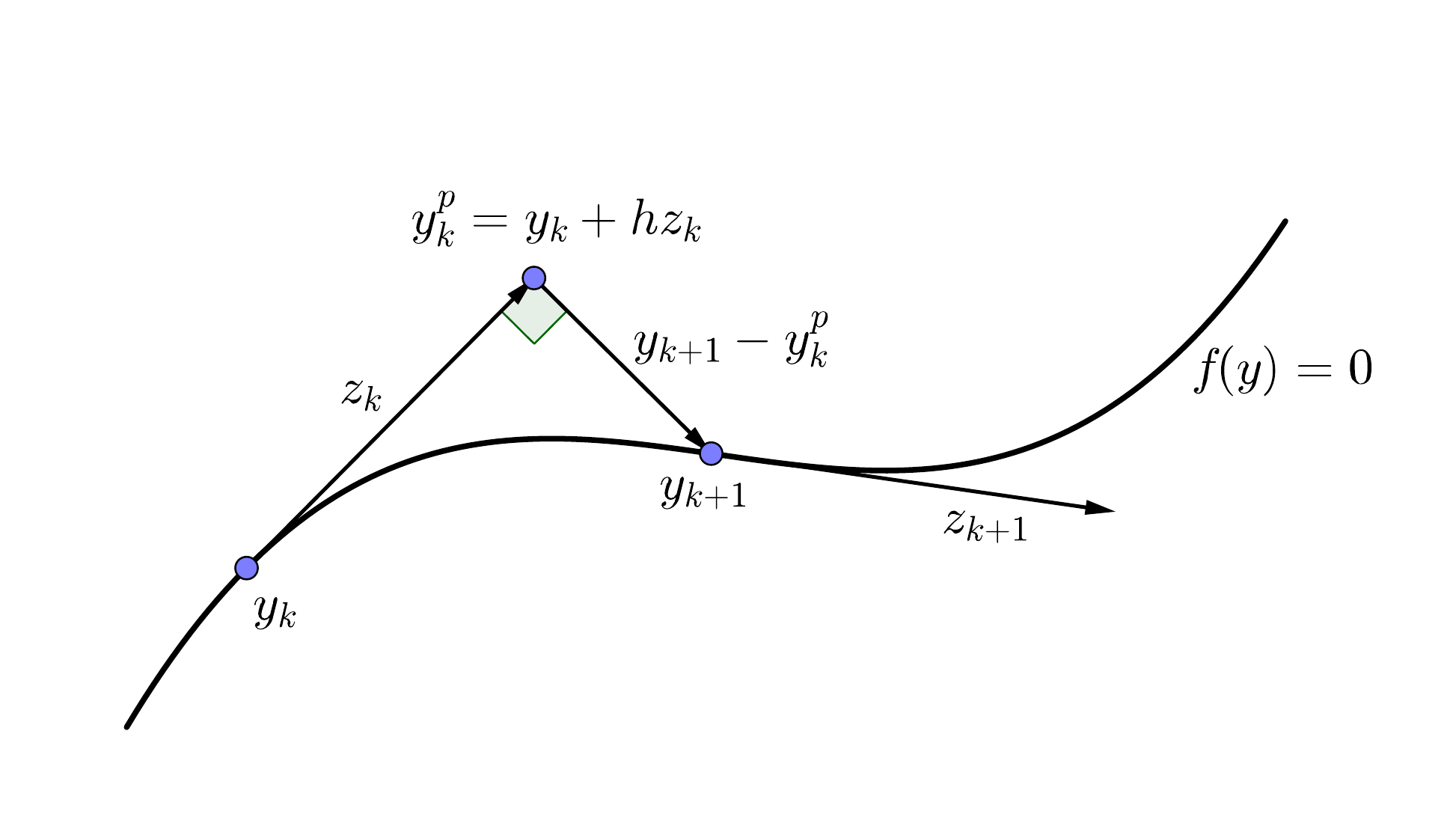}
\caption{ \small
Illustration of the pseudo-arclength method: given $y_k$ such that $f(y_k)=0$, the method computes a point $y_{k+1}$ such that $f(y_{k+1})=0$ and a consistently-oriented tangent direction $z_{k+1}$ at $y_{k+1}$.
}\label{F:pseudoarclength}
\end{figure}

To apply the pseudo-arclength method to our problem, we fix $\tau _0 > 0$ small and use the local bifurcation formulas for $c(\tau )$ and $\varphi(x;\tau )$ provided in Proposition \ref{prop:local} to form
\[
y_0^* := (c(\tau _0), \vec{\varphi}_N(\tau _0)) \in \RM^{N+1},
\]
where $\vec{\varphi}_N = (\varphi_N^1,\ldots,\varphi_N^N) := (\varphi(x_1;\tau _0), \ldots, \varphi(x_N;\tau _0))$ and the $x_i = \dfrac{(2i-1)\pi}{2N}$, $i=1,\ldots,N$ are the collocation points on $[0,\pi]$.  Moreover, since the local bifurcation curve is parametrized by $\tau $, we compute the tangent direction to the local bifurcation curve at $y_0^*$ by differentiating with respect to $\tau $ at $\tau _0$ and normalizing to unit length:
\[
\widetilde{z}_0 := \left( c'(\tau _0), \, \frac{\partial \vec{\varphi}_N}{\partial \tau }(\tau _0) \right), \quad z_0 := \frac{\widetilde{z}_0}{|\widetilde{z}_0|},
\]
where
\begin{align*}
c'(\tau _0) &= \frac{3\tau _0}{4}\left[ -\frac{1}{c_k} + 3c_k \left( \frac{1}{c_k^2-1} + \frac{1}{2(c_k^2 - c_{2k}^2)} \right) \right] \\
\frac{\partial \vec{\varphi}_N}{\partial \tau }(\tau_0) &= ( \cos(x_1), \, \cos(x_2), \ldots, \cos(x_N) ).
\end{align*}
Now, $y_0^*$ does not necessarily satisfy  $f(y_0^*)=0$, hence we initially solve $f(y)=0$ via Newton's method
using $y_0^*$ as an initial guess to obtain $y_0$ such that $f(y_0)=0$.  For small $\tau_0$, this $y_0$ will be close to $y_0^*$ with the tangent direction at $y_0^*$ being a sufficiently adequate approximation of the tangent direction at $y_0$.  Thus we will use these $y_0$ and $z_0$ to seed the pseudo-arclength method, with $f : \RM^{N+1} \to \RM^N$ given by \eqref{E:pseudoarclength_f}.

The bifurcation diagram in Figure \ref{fig:bifcurve1} was generated taking $N=512$.  Note the monotonicity
of the wave profile, guaranteed by Lemma \ref{lemma:lower_blowup improved}, allows us to approximate
the waveheight as
\[
\text{waveheight} = \varphi(0)-\varphi(\pi) \approx \varphi_N(x_1) - \varphi_N(x_N),
\]
which is plotted against the corresponding wavespeed in Figure \ref{fig:bifcurve1}.  
We continued to run the pseudo-arclength method so long as $\phi_N(x_1)\approx\phi(0)$, defined above,
does not exceed the value $\gamma$.

Though it requires more computation than a more-traditional parameter continuation method, the pseudo-arclength method excels when generating bifurcation diagrams containing a turning point, as occurs near the top of the curve plotted in Figure \ref{fig:bifcurve1}.  In parameter continuation, one continues the collocation method by manually stepping the values of the parameter, which is problematic near turning points where the parametrization fails to be a function of the wavespeed parameter.  This difficulty can sometimes be side-stepped by switching parametrizations near the turning point.  However, the pseudo-arclength method is more robust in the sense that one does not have to manually decide how to step the parameter; the method computes the parameter and collocation values simultaneously.

Ideally, for solutions computed by the above continuation algorithm, one would run a time-evolution to ensure that the approximated profiles persist over multiples of the temporal period.  
However, we believe our time-evolution analysis failed due to the expected ill-posedness of the local dynamics about these sign-changing waves.
See Remark \ref{signdefinite} and \cite{CJ17} for further discussion.

\section{Extension to waves with sign-definite height profiles}\label{appendix_b}

In this section we describe the extensions necessary to prove the existence of a wave $\varphi$ that is strictly positive and else satisfies the conditions of Theorem~\ref{thm:main}.
As mentioned in Remark~\ref{rem_positivehighest}, it was shown in \cite[Section 3.1.3]{CJ17} that a one-parameter family of strictly positive, $2\pi$-periodic traveling waves $\varphi$ of 
the nonlocal profile equation \eqref{profile} bifurcates from the curve $\Gamma_-$ of trivial solutions at some $c_1 \approx 1.11834$.  This local bifurcation curve can be 
continued into a global curve as in Theorem~\ref{thm:global} by using precisely the same argument as in Section~\ref{sec:bifurcation} above. These waves satisfy the same nodal and regularity properties as the ones for \(c \in (0,1)\), so that they are smooth wherever \(\varphi < \gamma\), even and strictly rising on the half-period \((-\pi,0)\). 
According to Lemma~\ref{lemma:uniqueness} this global bifurcation curve cannot pass $c=1$ without crossing the curve of zero solutions. As the following Lemma shows this implies strict positivity of the solutions along it.  

\begin{lemma}\label{lemma:positivewave}
Solutions along the global bifurcation branch for \(c > 1\) are strictly positive .
\end{lemma}
\begin{proof}
As \(\Gamma_{-}(c) > 0\) for \(c > 1\), the solutions \(\varphi\) bifurcating off \(\Gamma_{-}\) at \(c_1 > 1\) are strictly positive for sufficiently small values of the bifurcation parameter \(\tau\). Define \(\tau_0\) to be the smallest positive value of \(\tau\) for which \(\min \varphi(\tau) = \varphi(\tau)|_{x = \pi} = 0\) (if it exists; if not, all solutions along the global bifurcation branch are positive). Evaluating the steady equation~\eqref{profile} at a point where \(\varphi = 0\) shows that \(\mathcal{K} \varphi = 0\) at that point. Because \(\varphi(\tau_0) \geq 0\) this implies that \(\varphi(\tau_0)\) must vanish identically. But that is only possible if the global bifurcation branch has returned to the line \(\{(0,c) \colon c \in \mathbb{R}\}\) of zero solutions, and indeed for \(c > 1\). That, in turn,  would contradict the uniqueness statement proved in Proposition~\ref{prop:local}: there are no bifurcation points \((0,c)\) for \(c > 1\). Hence, there does not exist a finite \(\tau_0\) as above, and all solutions \(\varphi(\tau)\) along the global bifurcation branch starting from \(\Gamma_{-}\) at \(c_1 > 0\) are strictly positive.
\end{proof}

To classify the limiting behavior of solutions at the end of the global bifurcation curve of strictly positive solutions one again proves that alternatives (i) and (iii) in Theorem~\ref{thm:global} are excluded.  Observe first that the curve $\max(\varphi)=\Gamma_-$ intersects the curve $\max(\varphi)=\gamma$ at some $c_\gamma > c_1$ satisfying $\Gamma_-(c_\gamma)=\gamma(c_\gamma)$.   The next
result shows the bifurcation curve cannot cross that value.

\begin{lemma}\label{lemma:uniqueness2}
When $c=c_\gamma$ the unique solution of~\eqref{profile} satisfying $\varphi\leq\gamma$ with $\varphi(0)=\gamma$ is the constant solution $\varphi=\gamma$.
\end{lemma}
\begin{proof}
At $c=c_\gamma$ the constant $\varphi=\gamma$ is a solution of the profile equation \eqref{profile}.  In particular, $N(\gamma)=\gamma$.
Consequently, if $\varphi(0)=\gamma$ it follows from the positivity and normalization of the kernel \(K\) that 
\[
N(\gamma) = N \circ \varphi(0) = \mathcal{K}\varphi (0) \leq  \varphi(0) = \gamma,
\]
where equality is only possible if \(\varphi\) is constant. Hence,  $\varphi=\gamma$ as claimed.
\end{proof}

Alternative (iii) in Theorem~\ref{thm:global} can then be excluded by noting that Lemma~\ref{lemma:noperiod} applies directly
to the present case, the only modification being that now $c(\bar\tau)\in(1,c_\gamma)$.  Combining Theorem~\ref{thm:nodal} and Lemma~\ref{lemma:noperiod} it follows
that all solutions $\varphi(\tau)$ along the global bifurcation curve starting from $\Gamma_-$ at $c=c_1$ are nontrivial, positive, smooth, even, and strictly increasing on $(-\pi,0)$.
Finally, Lemmas~\ref{lemma:arzela} and~\ref{lemma:c_uniformbound} can be directly adapted to the present case, consequently excluding alternative (i) from Theorem~\ref{thm:global}. Thus, only alternative (ii) can hold for the bifurcation branch under consideration. As in Section~\ref{sec:bifurcation} above this implies
the existence of a highest, $2\pi$-periodic traveling wave solution of \eqref{profile} that is strictly increasing on $(-\pi,0)$, smooth on $\RM\setminus 2\pi\ZM$, and satisfies
\[
\gamma-\varphi(x)\simeq\left|x\log|x|\right|,
\]
for all $|x|\ll 1$ sufficiently small. This highest wave is everywhere positive.


\vspace{1em}

{
\noindent
{\bf Acknowledgements.} 
The authors would like to thank Sandra Pott and Atanas Stefanov for useful conversations, in particular regarding appropriate function spaces for solutions of the equation~\eqref{profile}.
We are also indebted to two of the referees for their proof-reading and valuable comments that helped improve our paper.

{\small

\bibliographystyle{siam}

}

\end{document}